\documentclass[a4paper]{ijbc_la}
\usepackage{amsmath,amssymb}
\usepackage{amsfonts}
\usepackage{enumerate}
\usepackage{epsfig,afterpage}
\usepackage[dvips]{psfrag}
\usepackage{psfrag}
\usepackage{graphicx}
\usepackage[active]{srcltx}
\usepackage{multirow}
\usepackage{color}

\def\be{\begin{equation}}
\def\ee{\end{equation}}
\def\ba{\begin{array}}
\def\ea{\end{array}}
\def\bq{\begin{eqnarray}}
\def\eq{\end{eqnarray}}
\def\beq{\begin{eqnarray*}}
\def\eeq{\end{eqnarray*}}
\def\bi{\begin{itemize}}
\def\ei{\end{itemize}}
\def\bc{\begin{center}}
\def\ec{\end{center}}

\def\bdf{\begin{definition}}
\def\edf{\end{definition}}
\def\bal{\begin{aligned}}
\def\eal{\end{aligned}}
\def\bth{\begin{theorem}}
\def\eth{\end{theorem}}
\def\brm{\begin{remark}}
\def\erm{\end{remark}}
\def\la{\lambda}

\def\l{\ell}

\def\o{\omega}

\def\p{\partial}
\def\deg{\mbox{deg}}

\def\ov{\overline}

\newcommand{\is}{\mbox{\scriptsize{\sl \$}}}

\newcommand{\R}{\mathbb{R}}
\newcommand{\C}{\mathbb{C}}
\DeclareMathOperator{\SC}{\mathbb{S}}

\newcommand{\N}{\mathbb{N}}

\newcommand{\PP}{\mathbb{P}}

\title{The geometry of quadratic polynomial differential systems with a finite and an infinite saddle--node $(A,B)$}

\author{Joan C. Art\'es}
\address{Departament de Matem\`{a}tiques, Universitat Aut\`{o}noma de
Barcelona, 08193 Bellaterra, Barcelona, Spain\\E--mail: artes@mat.uab.cat}

\author{Alex C. Rezende$^1$ and Regilene D. S. Oliveira$^2$}
\address{Departamento de Matem\'{a}tica,
Universidade de S\~{a}o Paulo, \\13566--590, S\~{a}o Carlos,
S\~{a}o Paulo, Brazil, \\E--mail:
$^1$arezende@icmc.usp.br,
$^2$regilene@icmc.usp.br}

\date{}

\abstract{
Planar quadratic differential systems occur in many areas of applied mathematics. Although more than one thousand papers have been
written on these systems, a complete understanding of this family is still missing. Classical problems, and in particular, Hilbert's 16th problem \cite{Hilbert:1900,Hilbert:1902}, are still open for this family. Our goal is to make a global study of the family $Q\overline{sn}\overline{SN}$ of all real quadratic polynomial differential systems which have a finite semi--elemental saddle--node and an infinite saddle--node formed by the collision of two infinite singular points. This family can be divided into three different subfamilies, all of them with the finite saddle--node in the origin of the plane with the eigenvectors on the axes and $(A)$ with the infinite saddle--node in the horizontal axis, $(B)$ with the infinite saddle--node in the vertical axis and $(C)$ with the infinite saddle--node in the bisector of the first and third quadrants. These three subfamilies modulo the action of the affine group and time homotheties are three--dimensional and we give their bifurcation diagram with respect to a normal form, in the three--dimensional real space of the parameters of these forms. In this paper we provide the complete study of the geometry of the first two families, $(A)$ and $(B)$. The bifurcation diagram for the subfamily $(A)$ yields 29 phase portraits for systems in $Q\overline{sn}\overline{SN}(A)$ counting phase portraits with and without limit cycles, while the bifurcation diagram for the subfamily $(B)$ yields 16 phase portraits for systems in $Q\overline{sn}\overline{SN}(B)$ under the same conditions. Case $(C)$ will yield quite more cases and will have an independent paper in short. Algebraic invariants are used to construct the bifurcation set. The phase portraits are represented on the Poincar\'{e} disk. The bifurcation set of $Q\overline{sn}\overline{SN}(A)$ is not only algebraic due to the presence of a surface found numerically. All points in this surface correspond to connections of separatrices.}

\runningheads{J.C. Art\'{e}s, A.C. Rezende and R.D.S. Oliveira}{The geometry of quadratic polynomial differential systems with a finite and an infinite saddle--node $(A,B)$}

\newpage

\begin{document}

\maketitle\clearpage

\section{Introduction, brief review of the literature and statement of results}\label{sec:int}

\indent Here we call \textit{quadratic differential systems} or simply \textit{quadratic systems}, differential systems of the form
    \be \ba{lcccl}
            \dot{x}&=& p(x,y), \\
            \dot{y}&=& q(x,y), \\
    \ea \label{eq:qs} \ee
where $p$ and $q$ are polynomials over $\R$ in $x$ and $y$ such that the max(\deg$(p)$,\deg$(q))=2$. To such a system one can always associate the quadratic vector field
    \be
        X=p\frac{\p}{\p x}+q\frac{\p}{\p y} \label{eq:qvf},
    \ee
as well as the differential equation
    \be
        qdx-pdy=0.
        \label{eq:de}
    \ee
The class of all quadratic differential systems (or quadratic vector fields) will be denoted by $QS$.

We can also write system \eqref{eq:qs} as
    \be \ba{lcccl}
        \dot{x} & = p_0+p_{1}(x,y)+p_{2}(x,y)=p(x,y), \\
        \dot{y} & = q_0+q_{1}(x,y)+q_{2}(x,y)=q(x,y), \\
    \ea \label{2l1} \ee
where $p_i$ and $q_i$ are homogeneous polynomials of degree $i$ in $(x,y)$ with real coefficients with $p_{2}^2+q_{2}^2 \neq 0$.

Even after hundreds of studies on the topology of real planar quadratic vector fields, it is kind of impossible to outline a complete characterization of their phase portraits, and attempting to topologically classify them, which occur rather often in applications, is quite a complex task. This family of systems depends on twelve parameters, but due to the action of the group $G$ of real affine transformations and time homotheties, the class ultimately depends on five parameters, but this is still a large number.

This paper is aimed at studying the class $Q\overline{sn}\overline{SN}$ of all quadratic systems possessing a finite saddle--node $\overline{sn}_{(2)}$ and an infinite saddle--node of type $\overline{\!{0\choose 2}\!\!}\ SN$. The finite saddle--node is a semi--elemental point whose neighborhood is formed by the union of two hyperbolic sectors and one parabolic sector. By a semi--elemental point we understand a point with zero determinant of its Jacobian, but only one eigenvalue zero. These points are known in classical literature as semi--elementary, but we use the term semi--elemental introduced in \cite{Artes-Llibre-Schlomiuk-Vulpe:2012} as part of a set of new definitions more deeply related to singular points, their multiplicities and, specially, their Jacobian matrices. In addition, an infinite saddle--node of type $\overline{\!{0\choose 2}\!\!}\ SN$ is obtained by the collision of an infinite saddle with an infinite node. There is another type of infinite saddle--node denoted by $\overline{\!{1\choose 1}\!\!}\ SN$ which is given by the collision of a finite antisaddle (respectively, finite saddle) with an infinite saddle (respectively, infinite node) and which will appear in some of the phase portraits.

The condition of having a finite saddle--node of all the systems in $Q\overline{sn}\overline{SN}$ implies that these systems may have up to two other finite points.

For a general framework of study of the class of all quadratic differential systems we refer to the article of Roussarie and Schlomiuk \cite{Roussarie-Schlomiuk:2002}.

In this study we follow the pattern set out in \cite{Artes-Llibre-Schlomiuk:2006}. As much as possible we shall try to avoid repeating technical sections which are the same for both papers, referring to the paper mentioned just above, for more complete information.

In this article we give a partition of the classes $Q\overline{sn}\overline{SN}(A)$ and $Q\overline{sn}\overline{SN}(B)$. The first class $Q\overline{sn}\overline{SN}(A)$ is partitioned into 66 parts: 23 three--dimensional ones, 31 two--dimensional ones, 11 one--dimensional ones and 1 point. This partition is obtained by considering all the bifurcation surfaces of singularities, one related to the presence of another invariant straight line and one related to connections of separatrices, modulo ``islands'' (see Sec. \ref{sec:islands}). The second class $Q\overline{sn}\overline{SN}(B)$ is partitioned into 30 parts: 9 three--dimensional ones, 14 two--dimensional ones, 6 one--dimensional ones and 1 point, which are all algebraic and obtained by considering all the bifurcation surfaces.

A \textit{graphic} as defined in \cite{Dumortier-Roussarie-Rousseau:1994} is formed by a finite sequence of points $r_1,r_2,\ldots,r_n$ (with possible repetitions) and non--trivial connecting orbits $\gamma_i$ for $i=1,\ldots,n$ such that  $\gamma_i$ has $r_i$ as $\alpha$--limit set and $r_{i+1}$ as $\omega$--limit set for $i<n$ and $\gamma_n$ has $r_n$ as  $\alpha$--limit set and $r_{1}$ as $\omega$--limit set. Also normal orientations $n_j$ of the non--trivial orbits must be coherent  in the sense that if $\gamma_{j-1}$ has left--hand orientation then so does $\gamma_j$. A \textit{polycycle} is a graphic which has  a Poincar\'{e} return map. For more details, see \cite{Dumortier-Roussarie-Rousseau:1994}.

\onecolumn
\begin{figure}
\psfrag{V1}{$V_{1}$}   \psfrag{V3}{$V_{3}$}   \psfrag{V6}{$V_{6}$}
\psfrag{V9}{$V_{9}$}   \psfrag{XX}{$V_{11}$}  \psfrag{X12}{$V_{12}$}
\psfrag{X14}{$V_{14}$} \psfrag{X15}{$V_{15}$} \psfrag{X16}{$V_{16}$}
\psfrag{1S1}{$1S_{1}$} \psfrag{1S2}{$1S_{2}$} \psfrag{1S4}{$1S_{4}$}
\psfrag{1S5}{$1S_{5}$} \psfrag{3S1}{$3S_{1}$} \psfrag{3S2}{$3S_{2}$}
\psfrag{3S3}{$3S_{3}$} \psfrag{3S4}{$3S_{4}$} \psfrag{5S1}{$5S_{1}$}
\psfrag{5S2}{$5S_{2}$} \psfrag{5S3}{$5S_{3}$} \psfrag{7S1}{$7S_{1}$}
\psfrag{8S1}{$8S_{1}$} \psfrag{8S4}{$8S_{4}$} \psfrag{1.2L2}{$1.2L_{2}$}
\psfrag{1.8L1}{$1.8L_{1}$} \psfrag{2.3L1}{$2.3L_{1}$} \psfrag{3.5L1}{$3.5L_{1}$}
\psfrag{P1}{$P_{1}$}
\centerline{\psfig{figure=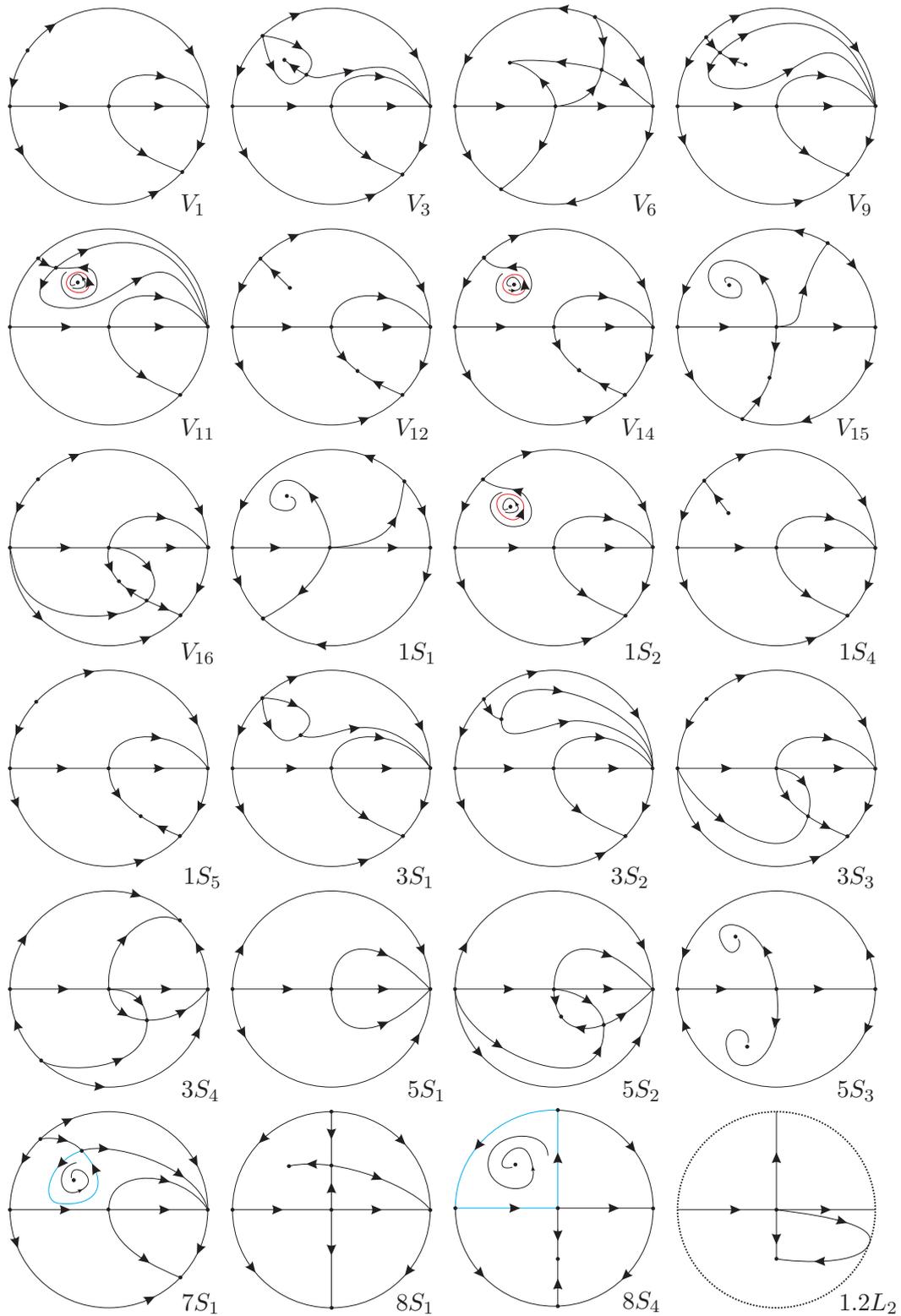,width=14cm}} \centerline {}
\caption{\small \label{fig:phase_a1} Phase portraits for quadratic vector fields with a finite saddle--node $\overline{sn}_{(2)}$ and an infinite saddle--node of type $\overline{\!{0\choose 2}\!\!}\ SN$ in the horizontal axis.}
\end{figure}
\twocolumn

\onecolumn
\begin{figure}
\psfrag{1.8L1}{$1.8L_{1}$} \psfrag{2.3L1}{$2.3L_{1}$} \psfrag{3.5L1}{$3.5L_{1}$}
\psfrag{3.8L1}{$3.8L_{1}$} \psfrag{P1}{$P_{1}$}
\centerline{\psfig{figure=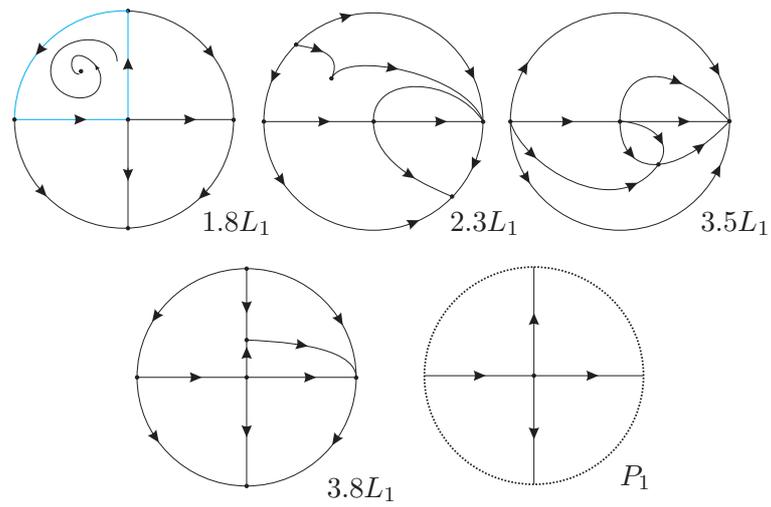,width=10cm}} \centerline {}
\caption{\small \label{fig:phase_a2} Continuation of Fig. \ref{fig:phase_a1}.}
\end{figure}

\begin{figure}
\psfrag{V1}{$V_{1}$}   \psfrag{V2}{$V_{2}$}   \psfrag{V3}{$V_{3}$} \psfrag{V6}{$V_{6}$}
\psfrag{V7}{$V_{7}$}   \psfrag{1S1}{$1S_{1}$} \psfrag{1S2}{$1S_{2}$}
\psfrag{1S3}{$1S_{3}$} \psfrag{1S4}{$1S_{4}$} \psfrag{2S1}{$2S_{1}$}
\psfrag{5S1}{$5S_{1}$} \psfrag{5S1}{$5S_{1}$} \psfrag{5S3}{$5S_{3}$}
\psfrag{1.2L1}{$1.2L_{1}$} \psfrag{1.2L2}{$1.2L_{2}$} \psfrag{1.5L1}{$1.5L_{1}$}
\psfrag{P1}{$P_{1}$}
\centerline{\psfig{figure=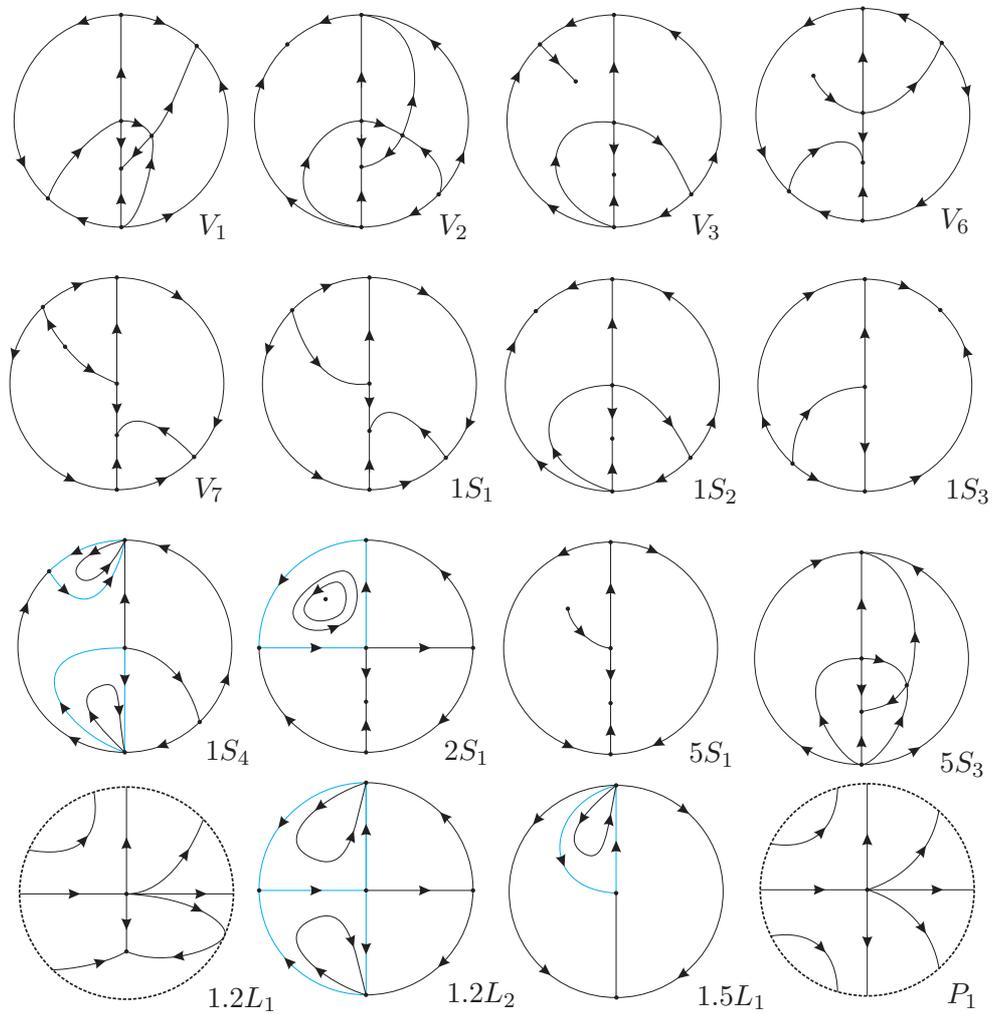,width=13cm}} \centerline {}
\caption{\small \label{fig:phase_b} Phase portraits for quadratic vector fields with a finite saddle--node $\overline{sn}_{(2)}$ and an infinite saddle--node of type $\overline{\!{0\choose 2}\!\!}\ SN$ in the vertical axis.}
\end{figure}
\twocolumn

\begin{theorem} \label{th:1.1} There exist 29 distinct phase portraits for the quadratic vector fields having a finite saddle--node $\overline{sn}_{(2)}$ and an infinite saddle--node of type $\overline{\!{0\choose 2}\!\!}\ SN$ located in the direction defined by the eigenvector with null eigenvalue (class $Q\overline{sn}\overline{SN}(A)$). All these phase portraits are shown in Figs. \ref{fig:phase_a1} and \ref{fig:phase_a2}. Moreover, the following statements hold:
\begin{enumerate}[(a)]
\item The manifold defined by the eigenvector with null eigenvalue is always an invariant straight line under the flow;
\item There exist three phase portraits with limit cycles, and they are in the regions $V_{11}$, $V_{14}$ and $1S_{2}$;
\item There exist three phase portraits with graphics, and they are in the regions $7S_{1}$, $8S_{4}$ and $1.8L_{1}$.
\end{enumerate}
\end{theorem}

\begin{theorem} \label{th:1.2} There exist 16 distinct phase portraits for the quadratic vector fields having a finite saddle--node $\overline{sn}_{(2)}$ and an infinite saddle--node of type $\overline{\!{0\choose 2}\!\!}\ SN$ located in the direction defined by the eigenvector with non--null eigenvalue (class $Q\overline{sn}\overline{SN}(B)$). All these phase portraits are shown in Fig. \ref{fig:phase_b}. Moreover, the following statements hold:
\begin{enumerate}[(a)]
\item The manifold defined by the eigenvector with non--null eigenvalue is always an invariant straight line under the flow;
\item There exist four phase portraits with graphics, and they are in the regions $1S_{4}$, $2S_{1}$, $1.2L_{2}$ and $1.5L_{1}$;
\item There exists one phase portrait with an integrable center, and it is in the region $2S_{1}$;
\item There exists one phase portrait with an integrable saddle, and it is in the region $2S_{2}$.
\end{enumerate}
\end{theorem}

For the class $Q\overline{sn}\overline{SN}(A)$, from its 29 different phase portraits, 9 occur in 3--dimensional parts, 14 in 2--dimensional parts, 5 in 1--dimensional parts and 1 occur in a single 0--dimensional part, and for the class $Q\overline{sn}\overline{SN}(B)$, from its 16 different phase portraits, 5 occur in 3--dimensional parts, 7 in 2--dimensional parts, 3 in 1--dimensional parts and 1 occur in a single 0--dimensional part.

In Figs. \ref{fig:phase_a1}, \ref{fig:phase_a2} and \ref{fig:phase_b} we have denoted all the singular points with a small disk. We have plotted with wide curves the separatrices and we have added some thinner orbits to avoid confusion in some required cases.

\begin{remark} \rm We label the phase portraits according to the parts of the bifurcation diagram where they occur. These labels could be different for two topologically equivalent phase portraits occurring in distinct parts. Some of the phase portraits in 3--dimensional parts also occur in some 2--dimensional parts bordering these 3--dimensional parts. An example occurs when a node turns into a focus. An analogous situation happens for phase portraits in 2--dimensional (respectively, 1--dimensional) parts, coinciding with a phase portrait on 1--dimensional (respectively, 0--dimensional) part situated on the border of it.
\end{remark}

The work is organized as follows. In Sec. \ref{sec:qvfsn2SN02ab} we describe the normal form for the families of systems having a finite saddle--node and an infinite saddle--node of type $\overline{\!{0\choose 2}\!\!}\ SN$ in both horizontal and vertical axes.

For the study of real planar polynomial vector fields two compactifications are used. In Sec. \ref{sec:poincare} we describe very briefly the Poincar\'{e} compactification on the 2--dimensional sphere.

In Sec. \ref{sec:basicwf} we list some very basic properties of general quadratic systems needed in this study.

In Sec. \ref{sec:internum} we mention some algebraic and geometric concepts that were introduced in \cite{Schlomiuk-Pal:2001,Llibre-Schlomiuk:2004} involving intersection numbers, zero--cycles, divisors, and T--comitants and invariants for quadratic systems as used by the Sibirskii school. We refer the reader directly to \cite{Artes-Llibre-Schlomiuk:2006} where these concepts are widely explained.

In Secs. \ref{sec:bifur_a} and \ref{sec:bifur_b}, using algebraic invariants and T--comitants, we construct the bifurcation surfaces for the classes $Q\overline{sn}\overline{SN}(A)$ and $Q\overline{sn}\overline{SN}(B)$, respectively.

In Sec. \ref{sec:islands} we comment about the possible existence of ``islands'' in the bifurcation diagram.

In Sec. \ref{sec:mainthm} we introduce a global invariant denoted by ${\cal I}$, which classifies completely, up to topological equivalence, the phase portraits we have obtained for the systems in the classes $Q\overline{sn}\overline{SN}(A)$ and $Q\overline{sn}\overline{SN}(B)$. Theorems \ref{t11} and \ref{t12} show clearly that they are uniquely determined (up to topological equivalence) by the values of the invariant ${\cal I}$.

\begin{remark} \rm It is worth mentioning that a third subclass $Q\overline{sn}\overline{SN}(C)$ of $Q\overline{sn}\overline{SN}$ must be considered. This subclass consists of planar quadratic systems with a finite saddle--node $\overline{sn}_{(2)}$ situated as in we do in this work and an infinite saddle--node of type $\overline{\!{0\choose 2}\!\!}\ SN$ in the bisector of the first and third quadrants and it is currently being studied by the same authors.
\end{remark}

In \cite{Artes-Kooij-Llibre:1998} the authors classified all the structurally stable quadratic planar systems modulo limit cycles, also known as the codimension--zero quadratic systems (roughly speaking, those systems whose all singularities, finite and infinite, are simple, with no separatrix connection, and where any nest of limit cycles is considered a single point with the stability of the outer limit cycle) by proving the existence of 44 topologically different phase portraits for these systems. The natural continuation in this idea is the classification of the structurally unstable quadratic systems of codimension--one, i.e. those systems which have one and only one of the following simplest structurally unstable objects: a saddle--node of multiplicity two (finite or infinite), a separatrix from one saddle point to another, and a separatrix forming a loop for a saddle point with its divergent non--zero. This study is already in progress \cite{Artes-Llibre:2013}, all topological possibilities have already been found, some of them have already been proved impossible and many representatives have been located, but still remain some cases without candidate. One way to obtain codimension--one phase portraits is considering a perturbation of known phase portraits of quadratic systems of higher degree of degeneracy. This perturbation would decrease the codimension of the system and we may find a representative for a topological equivalence class in the family of the codimension--one systems and add it to the existing classification.

In order to contribute to this classification, we study some families of quadratic systems of higher degree of degeneracy, e.g. systems with a weak focus of second order, see \cite{Artes-Llibre-Schlomiuk:2006}, and with a finite semi--elemental triple node, see \cite{Artes-Rezende-Oliveira:2013}. In this last paper, the authors show that, after a quadratic perturbation in the phase portrait $V_{11}$, the semi--elemental triple node is split into a node and a saddle--node and the new phase portrait is topologically equivalent to one of the topologically possible phase portrait of codimension one expected to exist.

The present study is part of this attempt of classifying all the codimension--one quadratic systems. We propose the study of a whole family of quadratic systems having a finite double saddle--node and an infinite saddle--node of type $\overline{\!{0\choose 2}\!\!}\ SN$. Both subfamilies reported here will not bifurcate to any of the codimension--one systems still missing, but in the subfamily $Q\overline{sn}\overline{SN}(C)$ will appear some new examples due to the fact that the complexity and richness of the bifurcation diagram will be the highest one we have already found until now.

\section{Quadratic vector fields with a finite saddle--node $\overline{sn}_{(2)}$ and an infinite saddle--node of type $\overline{\!{0\choose 2}\!\!}\ SN$}\label{sec:qvfsn2SN02ab}

\indent A singular point $r$ of a planar vector field $X$ in $\R^2$ is \textit{semi--elemental} if the determinant of the matrix of its linear part, $DX(r)$, is zero, but its trace is different from zero.

The following result characterizes the local phase portrait at a semi--elemental singular point.

\begin{proposition} \label{th2.19} \cite{Andronov-Leontovich-Gordon-Maier:1973,Dumortier-Llibre-Artes:2006}
Let $r=(0,0)$ be an isolated singular point of the vector field $X$ given by
    \begin{equation}
        \begin{array}{ccl}
            \dot{x} & = & M(x,y), \\
            \dot{y} & = & y + N(x,y), \\
        \end{array}
        \label{eqth2.19}
    \end{equation}
where $M$ and $N$ are analytic in a neighborhood of the origin starting with at least degree 2 in the variables $x$ and $y$. Let $y=f(x)$ be the solution of the equation $y + N(x,y) = 0$ in a neighborhood of the point $r=(0,0)$, and suppose that the function $g(x) = M(x,f(x))$ has the expression $g(x) = a x^\alpha + o(x^\alpha)$, where $\alpha \geq 2$ and $a \neq 0$. So, when $\alpha$ is odd, then $r=(0,0)$ is either an unstable multiple node, or a multiple saddle, depending if $a>0$, or $a<0$, respectively. In the case of the multiple saddle, the separatrices are tangent to the $x$--axis. If $\alpha$ is even, the $r=(0,0)$ is a multiple saddle--node, i.e. the singular point is formed by the union of two hyperbolic sectors with one parabolic sector. The stable separatrix is tangent to the positive (respectively, negative) $x$--axis at $r=(0,0)$ according to $a<0$ (respectively, $a>0$). The two unstable separatrices are tangent to the $y$--axis at $r=(0,0)$.
\end{proposition}

In the particular case where $M$ and $N$ are real quadratic polynomials in the variables $x$ and $y$, a quadratic system with a semi--elemental singular point at the origin can always be written into the form
    \begin{equation}
        \begin{array}{ccl}
            \dot{x} & = & g x^2 + 2h xy + k y^2, \\
            \dot{y} & = & y + \ell x^2 + 2m xy + n y^2. \\
        \end{array}
        \label{eqsemielemental}
    \end{equation}

By Proposition \ref{th2.19}, if $g \neq 0$, then we have a double saddle--node $\overline{sn}_{(2)}$, using the notation introduced in \cite{Artes-Llibre-Schlomiuk-Vulpe:2012}.

In the normal form above, we consider the coefficient of the terms $xy$ in both equations multiplied by 2 in order to make easier the calculations of the algebraic invariants we shall compute later.

We note that in the normal form \eqref{eqsemielemental} we already have a semi--elemental point at the origin and its eigenvectors are $(1,0)$ and $(0,1)$ which condition the possible positions of the infinite singular points.

We suppose that there exists a $\overline{\!{0\choose 2}\!\!}\ SN$ at some point at the infinity. If this point is different from either $[1:0:0]$ of the local chart $U_{1}$, or $[0:1:0]$ of the local chart $U_{2}$, after a reparametrization of the type $(x,y) \to (x,\alpha y)$, $\alpha \in \R$, this point can be replaced at $[1:1:0]$ of the local chart $U_{1}$, that is, at the bisector of the first and third quadrants. However, if $\overline{\!{0\choose 2}\!\!}\ SN$ is at $[1:0:0]$ or $[0:1:0]$, we cannot apply this change of coordinates and it requires an independent study for each one of the cases, which in turn are not equivalent themselves due to the position of the infinite saddle--node with respect to the eigenvectors of the finite saddle--node.

\subsection{The normal form for the subclass $Q\overline{sn}\overline{SN}(A)$}

The following result states the normal form for systems in $Q\overline{sn}\overline{SN}(A)$.

\begin{proposition} \label{normalform_a}
Every system with a finite semi--elemental double saddle--node $\overline{sn}_{(2)}$ and an infinite saddle--node of type $\overline{\!{0\choose 2}\!\!}\ SN$ located in the direction defined by the eigenvector with null eigenvalue can be brought via affine transformations and time rescaling to the following normal form
    \begin{equation}
        \begin{array}{ccl}
            \dot{x} & = & x^2 + 2h xy + k y^2, \\
            \dot{y} & = & y + xy + n y^2, \\
        \end{array}
        \label{eqsna}
    \end{equation}
where $h$, $k$ and $n$ are real parameters.
\end{proposition}

\begin{proof}
We start with system \eqref{eqsemielemental}. This system already has a finite semi--elemental double saddle--node at the origin (then $g \neq 0$) with its eigenvectors in the direction of the axes. The first step is to place the point $\overline{\!{0\choose 2}\!\!}\ SN$ at the origin of the local chart $U_{1}$ with coordinates $(w,z)$. For that, we must guarantee that the origin is a singularity of the flow in $U_{1}$, $$\ba{l} \dot{w} = \l+(-g+2m)w+(-2h+n)w^2-kw^3+wz, \\ \dot{z} = (-g-2hw-kw^2)z. \\ \ea$$ Then, we set $\l=0$ and, by analyzing the Jacobian of the former expression, we set $m=g/2$ in order to have the eigenvalue associated to the eigenvector on $z=0$ being null. Since $g \neq 0$, by a time rescaling, we can set $g=1$ and obtain the form \eqref{eqsna}.
\end{proof}

In view that the normal form \eqref{eqsna} involves the coefficients $h$, $k$ and $n$, which are real, the parameter space is $\R^3$ with coordinates $(h,k,n)$.

\begin{remark} \rm
After rescaling the parameters, we note that system \eqref{eqsna} is symmetric in relation to the real parameter $h$. Then, we shall only consider $h \geq 0$.
\end{remark}

\begin{remark} \rm
We note that $\{y=0\}$ is an invariant straight line under the flow of \eqref{eqsna}.
\end{remark}

\subsection{The normal form for the subclass $Q\overline{sn}\overline{SN}(B)$}

The following result gives the normal form for systems in $Q\overline{sn}\overline{SN}(B)$.

\begin{proposition} \label{normalform_b}
Every system with a finite semi--elemental double saddle--node $\overline{sn}_{(2)}$ and an infinite saddle--node of type $\overline{\!{0\choose 2}\!\!}\ SN$ located in the direction defined by the eigenvector with non-null eigenvalue can be brought via affine transformations and time rescaling to the following normal form
    \begin{equation}
        \begin{array}{ccl}
            \dot{x} & = & x^2 + 2h xy, \\
            \dot{y} & = & y + \l x^2 + 2m xy + 2h y^2, \\
        \end{array}
        \label{eqsnb}
    \end{equation}
where $h$, $\l$ and $m$ are real parameters.
\end{proposition}

\begin{proof}
Analogously to Proposition \ref{normalform_a}, we start with system \eqref{eqsemielemental}, but now we want to place the point $\overline{\!{0\choose 2}\!\!}\ SN$ at the origin of the local chart $U_{2}$. By following the same steps, we set $k=0$, $n=h/2$, $g=1$ and we obtain the form \eqref{eqsnb}.
\end{proof}

In view that the normal form \eqref{eqsna} involves the coefficients $h$, $\l$ and $m$, which are real, the parameter space is $\R^3$ with coordinates $(h,\l,m)$.

\begin{remark} \rm
After rescaling the parameters, we note that system \eqref{eqsnb} is symmetric in relation to the real parameter $h$. Then, we will only consider $h \geq 0$.
\end{remark}

\begin{remark} \rm
We note that $\{x=0\}$ is an invariant straight line under the flow of \eqref{eqsnb}.
\end{remark}

\section{The Poincar\'{e} compactification and the complex (real) foliation with singularities on $\C\PP^2$ ($\R\PP^2$)} \label{sec:poincare}

\indent A real planar polynomial vector field $\xi$ can be compactified on the sphere as follows. Consider the $x$, $y$ plane as being the plane $Z=1$ in the space $\R^3$ with coordinates $X$, $Y$, $Z$. The central projection of the vector field $\xi$ on the sphere of radius one yields a diffeomorphic vector field on the upper hemisphere and also another vector field on the lower hemisphere. There exists (for a proof see \cite {Gonzales:1969}) an analytic vector field $cp(\xi)$ on the whole sphere such that its restriction on the upper hemisphere has the same phase curves as the one constructed above from the polynomial vector field. The projection of the closed northern hemisphere $H^+$ of $\,\SC^2$ on $Z=0$ under $(X,Y,Z) \to (X,Y)$ is called \textit{the Poincar\'{e} disc}. A singular point $q$ of $cp(\xi)$ is called an \textit{infinite} (respectively, \textit{finite}) singular point if $q\in \SC^1$, the equator (respectively, $q\in \SC^2 \setminus \SC^1$). By the \textit{Poincar\'{e} compactification of a polynomial vector field} we mean the vector field $cp(\xi)$ restricted to the upper hemisphere completed with the equator.

Ideas in the remaining part of this section go back to Darboux's work \cite{Darboux:1878}. Let $p(x,y)$ and $q(x,y)$ be polynomials with real coefficients. For the vector field
    \begin{equation}\label{21}
        p\frac{\p}{\p x}+ q\frac{\p}{\p y},
    \end{equation}
or equivalently for the differential system
    \begin{equation}\label{22}
        \dot x = p(x,y), \qquad \dot y= q(x,y),
    \end{equation}
we consider the associated differential $1$--form \linebreak $\o_{1}= q(x,y)dx- p(x,y)dy$, and the differential equation
    \begin{equation}\label{23}
        \o_{1}=0 \ .
    \end{equation}
Clearly, equation \eqref{23} defines a foliation with singularities on $\C^2$. The affine plane $\C^2$ is compactified on the complex projective space $\C\PP^2= (\C^3\setminus \{0\})/\sim$, where $(X,Y,Z)\sim (X',Y',Z')$ if and only if $(X,Y,Z)= \la (X',Y',Z')$ for some complex $\la\ne 0$. The equivalence class of $(X,Y,Z)$ will be denoted by $[X:Y:Z]$.

The foliation with singularities defined by equation \eqref{23} on $\C^2$ can be extended to a foliation with singularities on $\C\PP^2$ and the $1$--form $\o_{1}$ can be extended to a meromorphic $1$--form $\o$ on $\C\PP^2$ which yields an equation $\o=0$, i.e.
    \begin{equation}\label{ja1}
        A(X,Y,Z)dX+B(X,Y,Z)dY+C(X,Y,Z)dZ=0,
    \end{equation}
whose coefficients $A$, $B$, $C$ are homogeneous polynomials of the same degree and satisfy the relation:
    \begin{equation}\label{ja2}
        A(X,Y,Z)X+B(X,Y,Z)Y+C(X,Y,Z)Z=0,
    \end{equation}
Indeed, consider the map $i: \C^3 \setminus \{Z = 0\} \to \C^2$, given by $i(X,Y,Z)= (X/Z,Y/Z)=(x,y)$ and suppose that $\max\{\deg(p),\deg(q)\}= m>0$. Since $x=X/Z$ and $y=Y/Z$ we have:
    \[
        dx= (ZdX-XdZ)/Z^2, \quad dy= (ZdY-YdZ)/Z^2,
    \]
the pull--back form $i^*(\o_{1})$ has poles at $Z=0$ and yields the equation
    \[\aligned
        i^*(\o_{1})=& q(X/Z,Y/Z) (ZdX-XdZ)/Z^2 \\
                    & -p(X/Z,Y/Z) (ZdY-YdZ)/Z^2 = 0.
    \endaligned \]
Then, the $1$--form $\o= Z^{m+2} i^*(\o_{1})$ in $\C^3\setminus \{Z\ne 0\}$ has homogeneous polynomial coefficients of degree $m+1$, and for $Z=0$ the equations $\o=0$ and $i^*(\o_{1})=0$ have the same solutions. Therefore, the differential equation $\o=0$ can be written as \eqref{ja1}, where
    \be\aligned\label{24}
        A(X,Y,Z) =& Z Q(X,Y,Z)= Z^{m+1} q(X/Z,Y/Z),\\
        B(X,Y,Z) =& -Z P(X,Y,Z)=-Z^{m+1} p(X/Z,Y/Z),\\
        C(X,Y,Z) =& Y P(X,Y,Z)- X Q(X,Y,Z).\\
    \endaligned\ee

Clearly $A$, $B$ and $C$ are homogeneous polynomials of degree $m+1$ satisfying \eqref{ja2}.

In particular, for our quadratic systems \eqref{eqsna}, $A$, $B$ and $C$ take the following forms
    \begin{equation} \label{ja3}
        \begin{aligned}
            A(X,Y,Z)=& YZ(X-nY+Z)\\
            B(X,Y,Z)=& -(X^2+2hXY+kY^2)Z,\\
            C(X,Y,Z)=& Y(2hXY-nXY+kY^2-XZ).\\
        \end{aligned}
    \end{equation}
and for our quadratic systems \eqref{eqsnb}, $A$, $B$ and $C$ take the following forms
    \begin{equation} \label{ja4}
        \begin{aligned}
            A(X,Y,Z)=& Z(\l X^2+mXY+2hY^2+YZ),\\
            B(X,Y,Z)=& -X(X+2hY)Z,\\
            C(X,Y,Z)=& X-X(\l X^2-XY+2mXY+YZ).\\
        \end{aligned}
    \end{equation}

We note that the straight line $Z=0$ is always an algebraic invariant curve of this foliation and that its singular points are the solutions of the system: $A(X,Y,Z)= B(X,Y,Z)= C(X,Y,Z)=0$. We note also that $C(X,Y,Z)$ does not depend on $b$.

To study the foliation with singularities defined by the differential equation \eqref{ja1} subject to \eqref{ja2} with $A$, $B$, $C$ satisfying the above conditions in the neighborhood of the line $Z=0$, we consider the two charts of $\C\PP^2$: $(u,z)= (Y/X,Z/X)$, $X\ne 0$, and $(v,w)= (X/Y,Z/Y)$, $Y\ne 0$, covering this line. We note that in the intersection of the charts $(x,y)= (X/Z,Y/Z)$ and $(u,z)$ (respectively, $(v,w)$) we have the change of coordinates $x=1/z$, $y=u/z$ (respectively, $x=v/w$, $y=1/w$). Except for the point $[0:1:0]$ or the point $[1:0:0]$, the foliation defined by equations \eqref{ja1},\eqref{ja2} with $A$, $B$, $C$ as in \eqref{24} yields in the neighborhood of the line $Z=0$ the foliations associated with the systems
    \be \begin{aligned}
        \dot u =& uP(1,u,z)-Q(1,u,z)= C(1,u,z), \\
        \dot z=& z P(1,u,z), \label{26}
    \end{aligned} \ee
or
    \be \begin{aligned}
        \dot v =& vQ(v,1,w)-P(v,1,w)= -C(v,1,w), \\
        \dot w =& w P(v,1,w). \label{27}
    \end{aligned} \ee

In a similar way we can associate a real foliation with singularities on $\R\PP^2$ to a real planar polynomial vector field.

\section{A few basic properties of quadratic systems relevant for this study} \label{sec:basicwf}

\indent We list below results which play a role in the study of the global phase portraits of the real planar quadratic systems \eqref{eq:qs} having a semi--elemental triple node.

The following results hold for any quadratic system:
    \begin{enumerate}[(i)]
        \item \label{item_i} A straight line either has at most two (finite) contact points with a quadratic system (which include the singular points), or it is formed by trajectories of the system; see Lemma 11.1 of \cite{Ye:1986}. We recall that by definition a {\it contact point} of a straight line $L$ is a point of $L$ where the vector field has the same direction as $L$, or it is zero.

        \item \label{item_ii} If a straight line passing through two real finite singular points $r_{1}$ and $r_{2}$ of a quadratic system is not formed by trajectories, then it is divided by these two singular points in three segments $\ov{\infty r_{1}}$, $\ov{r_{1} r_{2}}$ and $\ov{r_{2} \infty}$ such that the trajectories cross $\ov{\infty r_{1}}$ and $\ov{r_{2}\infty}$ in one direction, and they cross $\ov{r_{1} r_{2}}$ in the opposite direction; see Lemma 11.4 of \cite{Ye:1986}.

        \item \label{item_iii} If a quadratic system has a limit cycle, then it surrounds a unique singular point, and this point is a focus; see \cite{Coppel:1966}.

        \item \label{item_iv} A quadratic system with an invariant straight line has at most one limit cycle; see \cite{Coll-Llibre:1988}.

        \item \label{item_v} A quadratic system with more than one invariant straight line has no limit cycle; see \cite{Bautin:1954}.
    \end{enumerate}

\begin{proposition}\label{th:4.2} Any graphic or degenerate graphic in a real planar polynomial differential system must either
\begin{enumerate}[1)]
\item surround a singular point of index greater than or equal to +1, or
\item contain a singular point having an elliptic sector situated in the region delimited by the graphic, or
\item contain an infinite number of singular points.
\end{enumerate}
\end{proposition}

\begin{proof} See the proof in \cite{Artes-Kooij-Llibre:1998}.
\end{proof}

\section{Some algebraic and geometric concepts} \label{sec:internum}

\indent In this article we use the concept of intersection number for curves (see \cite{Fulton:1969}). For a quick summary see Sec. 5 of \cite{Artes-Llibre-Schlomiuk:2006}.

We shall also use the concepts of zero--cycle and divisor (see \cite{Hartshorne:1977}) as specified for quadratic vector fields in \cite{Schlomiuk-Pal:2001}. For a quick summary see Sec. 6 of \cite{Artes-Llibre-Schlomiuk:2006}.

We shall also use the concepts of algebraic invariant and T--comitant as used by the Sibirskii school for differential equations. For a quick summary see Sec. 7 of \cite{Artes-Llibre-Schlomiuk:2006}.

In the next two sections we describe the algebraic invariants and T--comitants which are relevant in the study of families \eqref{eqsna}, see Sec. \ref{sec:bifur_a},  and \eqref{eqsnb}, see Sec. \ref{sec:bifur_b}.

\section{The bifurcation diagram of the systems in $Q\overline{sn}\overline{SN}(A)$} \label{sec:bifur_a}

We recall that, in view that the normal form \eqref{eqsna} involves the coefficients $h$, $k$ and $n$, which are real, the parameter space is $\R^3$ with coordinates $(h,k,n)$.

\subsection{Bifurcation surfaces due to the changes in the nature of singularities}

For systems \eqref{eqsna} we will always have $(0,0)$ as a finite singular point, a double saddle--node.

From Sec. 7 of \cite{Artes-Llibre-Vulpe:2008} we get the formulas which give the bifurcation surfaces of singularities in $\R^{12}$, produced by changes that may occur in the local nature of finite singularities. From \cite{Schlomiuk-Vulpe:2005} we get equivalent formulas for the infinite singular points. These bifurcation surfaces are all algebraic and they are the following:

\medskip

\noindent \textbf{Bifurcation surfaces in $\R^3$ due to multiplicities of singularities}

\medskip

\noindent {\bf (${\cal S}_{1}$)} This is the bifurcation surface due to multiplicity of infinite singularities as detected by the coefficients of the divisor $D_\R(P,Q;Z)= \sum_{W\in \{Z=0\}\cap \C\PP^2} I_W (P,Q) W$, (here $I_W (P,Q)$ denotes the intersection multiplicity of $P=0$ with $Q=0$ at the point $W$ situated on the line at infinity, i.e. $Z=0$) whenever $deg((D_\R(P,Q;Z)))>0$. This occurs when at least one finite singular point collides with at least one infinite point. More precisely this happens whenever the homogenous polynomials of degree two, $p_2$ and $q_2$, in $p$ and $q$ have a common root. In other words whenever \linebreak $\mu=$ Res$_x(p_2,q_2)/y^4= 0$. The equation of this surface is $$\mu = k - 2hn + n^2 = 0.$$

\medskip

\noindent {\bf (${\cal S}_{3}$)}\footnote{The numbers attached to these bifurcations surfaces do not appear here in increasing order. We just kept the same enumeration used in \cite{Artes-Llibre-Schlomiuk:2006} to maintain coherence even though some of the numbers in that enumeration do not occur here.} Since this family already has a saddle--node at the origin, the invariant $\mathbb{D}$, as defined in \cite{Artes-Llibre-Schlomiuk:2006}, is always zero. Since our normal form does not allow the existence of a finite singular point of multiplicity 3, the only possible bifurcation related to collisions of finite singularities with themselves is whether the other two finite singularities are either real, or complex, or form a double one. This phenomenon is captured by the T--comitant $\mathbb{T}$ as proved in \cite{Artes-Llibre-Vulpe:2008}. The equation of this surface is $$\mathbb{T} = h^2 - k = 0.$$

\medskip

\noindent {\bf (${\cal S}_{5}$)} Since this family already has a saddle--node at infinity, the invariant $\eta$, as defined in \cite{Artes-Llibre-Schlomiuk:2006}, is always zero. In this sense, we have to consider a bifurcation related to the existence of either the double infinite singularity $\overline{\!{0\choose 2}\!\!}\ SN$ plus a simple one, or a triple one. This phenomenon is ruled by the T--comitant $M$ as proved in \cite{Schlomiuk-Vulpe:2005,Artes-Llibre-Schlomiuk-Vulpe:2012}. The equation of this surface is $$M = 2h - n = 0.$$

\medskip

\noindent \textbf{The surface of $C^{\infty}$ bifurcation points due to a strong saddle or a strong focus changing the sign of their traces (weak saddle or weak focus)}

\medskip

\noindent {\bf (${\cal S}_{2}$)} This is the bifurcation surface due to the weakness of finite singularities, which occurs when the trace of a finite singular point is zero. The equation of this surface is given by $$\mathcal{T}_{4} = -4h^2 + 4k + n^2 = 0,$$ where $\mathcal{T}_4$ is defined in \cite{Vulpe:2011}. This $\mathcal{T}_4$ is an invariant.

This bifurcation can produce a topological change if the weak point is a focus or just a $C^\infty$ change if it is a saddle, except when this bifurcation coincides with a loop bifurcation associated with the same saddle, in which case, the change may also be topological.

\medskip

\noindent \textbf{The surface of $C^{\infty}$ bifurcation due to a node becoming a focus}

\medskip

\noindent {\bf (${\cal S}_{6}$)} This surface will contain the points of the parameter space where a finite node of the system turns into a focus. This surface is a $C^{\infty}$ but not a topological bifurcation surface. In fact, when we only cross the surface (${\cal S}_{6}$) in the bifurcation diagram, the topological phase portraits do not change. However, this surface is relevant for isolating the regions where a limit cycle surrounding an antisaddle cannot exist. Using the results of \cite{Artes-Llibre-Vulpe:2008}, the equation of this surface is given by $W_{4} = 0$, where $$\begin{aligned} W_{4} =& -48h^4 + 32h^{2}k + 16k^{2} + 64h^{3}n\\ & -64hkn - 24h^{2}n^{2} + 24kn^{2} + n^{4}.\\ \end{aligned}$$

\noindent \textbf{Bifurcation surface in $\R^3$ due to the presence of another invariant straight line}

\medskip

\noindent {\bf (${\cal S}_{8}$)} This surface will contain the points of the parameter space where another invariant straight line appears apart from $\{y=0\}$. This surface is split in some regions. Depending on these regions, the straight line may contain connections of separatrices from different saddles or not. So, in some cases, it may imply a topological bifurcation and, in others, just a $C^{\infty}$ bifurcation. The equation of this surface is given by $$\operatorname{Het} = h = 0.$$

\medskip

These, except (${\cal S}_{8}$), are all the bifurcation surfaces of singularities of systems \eqref{eqsna} in the parameter space and they are all algebraic. We shall discover another bifurcation surface not necessarily algebraic and on which the systems have global connection of separatrices different from that given by (${\cal S}_{8}$). The equation of this bifurcation surface can only be determined approximately by means of numerical tools. Using arguments of continuity in the phase portraits we can prove the existence of this not necessarily algebraic component in the region where it appears, and we can check it numerically. We will name it the surface (${\cal S}_{7}$).

%\begin{figure}
%\psfrag{ch}{$h$} \psfrag{ck}{$k$} \psfrag{cn}{$n$}
%\centerline{\psfig{figure=W4.eps,width=8cm}} \centerline {}
%\caption{\small \label{fig:3d2} The 3--dimensional picture of the surface (${\cal S}_{6}$) (when a finite node becomes a focus).}
%\end{figure}

\begin{remark} \rm
Even though we can draw a 3--dimensional picture of the algebraic bifurcation surfaces of singularities in $\R^3$, it is pointless to try to see a single 3--dimensional image of all these bifurcation surfaces together in the space $\R^3$. As we shall see later, the full partition of the parameter space obtained from all these bifurcation surfaces has 66 parts. %(see Fig. \ref{fig:3d2} for an example), it is pointless to try to see a single 3--dimensional image of all these bifurcation surfaces together in the space $\R^3$. As we shall see later, the full partition of the parameter space obtained from all these bifurcation surfaces has 66 parts.
\end{remark}

By the previous remark we shall foliate the 3--dimensional bifurcation diagram in $\R^3$ by planes $h=h_{0}$, $h_{0}$ constant. We shall give pictures of the resulting bifurcation diagram on these planar sections on an affine chart on $\R^2$. In order to detect the key values for this foliation, we must find the values of parameters where the surfaces have singularities and/or intersect to each other. As we mentioned before, we will be only interested in non-negative values of $h$ to construct the bifurcation diagram.

The following set of sixteen results study the singularities of each surface and the simultaneous intersection points of the bifurcation surfaces, or the points or curves where two bifurcation surfaces are tangent.

As the final bifurcation diagram is quite complex, it is useful to introduce colors which will be used to talk about the bifurcation points:

    \begin{enumerate}[(a)]
        \item the curve obtained from the surface (${\cal S}_{1}$) is drawn in blue (a finite singular point collides with an infinite one);

        \item the curve obtained from the surface (${\cal S}_{2}$) is drawn in yellow (when the trace of a singular point becomes zero);

        \item the curve obtained from the surface (${\cal S}_{3}$) is drawn in green (two finite singular points collide);

        \item the curve obtained from the surface (${\cal S}_{5}$) is drawn in red (three infinite singular points collide);

        \item the curve obtained from the surface (${\cal S}_{6}$) is drawn in black (an antisaddle is on the edge of turning from a node to a focus or vice versa);

        \item the curve obtained from the surface (${\cal S}_{7}$) is drawn in purple (the connection of separatrices); and

        \item the curve obtained from the surface (${\cal S}_{8}$) is also drawn in purple (presence of an invariant straight line). We draw it as a continuous curve if it implies a topological change or as a dashed curve if not.
    \end{enumerate}

We use the same color for (${\cal S}_{7}$) and (${\cal S}_{8}$) since both surfaces deal with connections of separatrices mostly.

\begin{lemma}\label{lem:lemma1A} Concerning the singularities of the surfaces, it follows that:
    \begin{enumerate}[(i)]
        \item (${\cal S}_{1}$), (${\cal S}_{2}$), (${\cal S}_{3}$), (${\cal S}_{5}$) and (${\cal S}_{8}$) have no singularities;

        \item (${\cal S}_{6}$) has singularities at the point $(0,0,0)$ and along the straight line $(h,0,2h)$.
    \end{enumerate}
\end{lemma}

\begin{proof} It is easy to see that the gradient vectors of each one of the surfaces (${\cal S}_{1}$), (${\cal S}_{2}$), (${\cal S}_{3}$), (${\cal S}_{5}$) and (${\cal S}_{8}$) are never null for all $(h,k,n) \in \R^3$; so \emph{(i)} is proved. In order to prove \emph{(ii)} we compute the gradient of $W_{4}$ and we verify that it is null whenever $h=k=n=0$ and along the straight line $(h,0,2h)$.
\end{proof}

\begin{lemma}\label{lem:lemma2A} The surfaces (${\cal S}_{1}$) and (${\cal S}_{2}$) intersect along the straight line $(h,0,2h)$ and along the curve $(h,8h^{2}/9,2h/3)$, for all $h \in \R$.
\end{lemma}

\begin{proof} By solving simultaneously both equations of the surfaces (${\cal S}_{1}$) and (${\cal S}_{2}$) for all $h \neq 0$, we obtain the straight line $(h,0,2h)$ and the curve $(h,8h^{2}/9,2h/3)$; if $h=0$, the only simultaneous solution is the origin.
\end{proof}

The proofs of the next lemmas are analogous to Lemma \ref{lem:lemma2A}, except when a different proof is included.

\begin{lemma}\label{lem:lemma3A} If $h=0$, the surfaces (${\cal S}_{1}$) and (${\cal S}_{3}$) intersect at the origin and, if $h \neq 0$, both surfaces intersect along the curve $(n,n^{2},n)$, $n \in \R$.
\end{lemma}

\begin{lemma}\label{lem:lemma4A} For all $h \in \R$, the surfaces (${\cal S}_{1}$) and (${\cal S}_{5}$) intersect along the straight line $(h,0,2h)$.
\end{lemma}

\begin{lemma}\label{lem:lemma5A} For all $h \in \R$, the surfaces (${\cal S}_{1}$) and (${\cal S}_{6}$) intersect along the straight line $(h,0,2h)$ and along the curve $(h,48h^{2}/49,6h/7)$.
\end{lemma}

\begin{lemma}\label{lem:lemma6A} For all $h \in \R$, the surfaces (${\cal S}_{1}$) and (${\cal S}_{8}$) intersect along the straight line $(h,0,2h)$.
\end{lemma}

\begin{lemma}\label{lem:lemma7A} For all $h \in \R$, the surfaces (${\cal S}_{2}$) and (${\cal S}_{3}$) intersect along the curve $(h,h^{2},0)$.
\end{lemma}

\begin{lemma}\label{lem:lemma8A} For all $h \in \R$, the surfaces (${\cal S}_{2}$) and (${\cal S}_{5}$) intersect along the straight line $(h,0,2h)$.
\end{lemma}

\begin{lemma}\label{lem:lemma9A} For all $h \in \R$, the surfaces (${\cal S}_{2}$) and (${\cal S}_{6}$) intersect along the straight line $(h,0,2h)$ and along the curve $(h,h^{2},0)$.
\end{lemma}

\begin{lemma}\label{lem:lemma10A} For all $h \in \R$, the surfaces (${\cal S}_{2}$) and (${\cal S}_{8}$) intersect along the straight lines $(h,0,-2h)$ and $(h,0,2h)$.
\end{lemma}

\begin{lemma}\label{lem:lemma11A} For all $h \in \R$, the surfaces (${\cal S}_{3}$) and (${\cal S}_{5}$) intersect along the curve $(h,h^{2},2h)$.
\end{lemma}

\begin{lemma}\label{lem:lemma12A} For all $h \in \R$, the surfaces (${\cal S}_{3}$) and (${\cal S}_{6}$) intersect along the curve $(h,h^{2},0)$.
\end{lemma}

\begin{lemma}\label{lem:lemma13A} If $h=0$, the surfaces (${\cal S}_{3}$) and (${\cal S}_{8}$) intersect along the straight line $(0,0,n)$, $n \in \R$, and, if $h \neq 0$, they have no intersection.
\end{lemma}

\begin{proof} By restricting the equations of both surfaces to $h=0$ and solving them simultaneously we obtain the straight line $(0,0,n)$, $n \in \R$. For all $h \neq 0$, the equation have no simultaneous solutions.
\end{proof}

\begin{lemma}\label{lem:lemma14A} For all $h \in \R$, the surfaces (${\cal S}_{5}$) and (${\cal S}_{6}$) intersect along the straight line $(h,0,2h)$.
\end{lemma}

\begin{lemma}\label{lem:lemma15A} For all $h \in \R$, the surfaces (${\cal S}_{5}$) and (${\cal S}_{8}$) intersect along the straight line $(h,0,2h)$.
\end{lemma}

\begin{lemma}\label{lem:lemma16A} For all $h \in \R$, the surfaces (${\cal S}_{6}$) and (${\cal S}_{8}$) intersect along the straight lines $(h,0,-6h)$ and $(h,0,2h)$.
\end{lemma}

Now we shall study the bifurcation diagram having as reference the values of $h$ where significant phenomena occur in the behavior of the bifurcation surfaces. As there is not any other critical value of $h$, except $h=0$, this is the only value where the behavior of the bifurcation surfaces changes critically. Recalling we are considering only non-negative values of $h$, we shall choose a positive value to be a generic case.

We take then the values:
    \begin{equation}
        \aligned
            h_{0} & = 0,\\
            h_{1} & = 1.\\
        \endaligned
    \end{equation}

The value $h_{0}$ corresponds to an explicit value of $h$ for which there is a bifurcation in the behavior of the systems on the slices. The value $h_{1}$ is just an intermediate point we call by a generic value of $h$ (see Figs. \ref{sliceh0A} and \ref{sliceh1A_nopurple}).

\begin{figure}
\centering
\psfrag{n}{$n$} \psfrag{k}{$k$}
\centerline{\psfig{figure=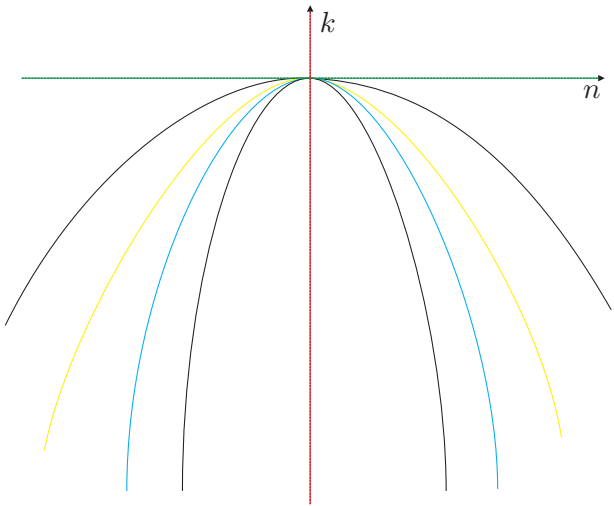,width=8cm}} \centerline {}
\caption{\small \label{sliceh0A} Slice of the parameter space for \eqref{eqsna} when $h=0$.}
\end{figure}

\begin{figure}
\centering
\psfrag{n}{$n$} \psfrag{k}{$k$}
\psfrag{1s2}{\small{$1s_{2}$}} \psfrag{1s3}{$1s_{3}$}
\psfrag{2s2}{$2s_{2}$} \psfrag{2s3}{$2s_{3}$} \psfrag{2s4}{$2s_{4}$}
\psfrag{6s2}{$6s_{2}$} \psfrag{8s2}{$8s_{2}$} \psfrag{8s3}{$8s_{3}$}
\psfrag{8s4}{$8s_{4}$} \psfrag{v4}{$v_{4}$} \psfrag{v5}{$v_{5}$}
\psfrag{v10}{$v_{10}$} \psfrag{v14}{$v_{14}$} \psfrag{v1}{$v_{1}$}
\psfrag{2.3l1}{$2.3\l_{1}$} \psfrag{1.8l1}{$1.8\l_{1}$} \psfrag{v3}{$v_{3}$}
\psfrag{3s1}{$3s_{1}$}
\centerline{\psfig{figure=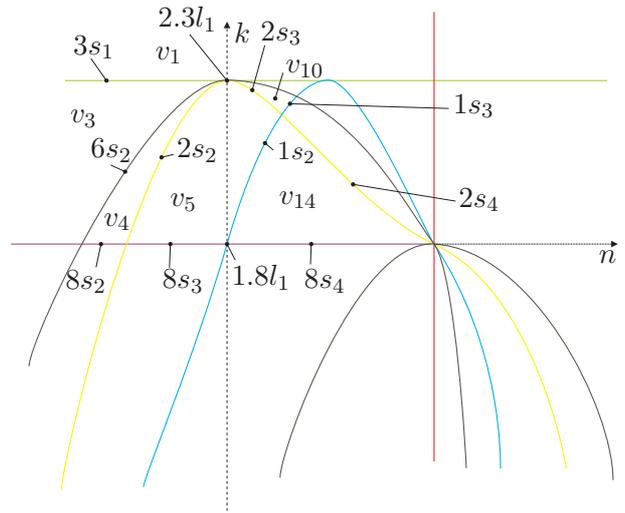,width=8cm}} \centerline {}
\caption{\small \label{sliceh1A_nopurple} Slice of the parameter space for \eqref{eqsna} when $h=1$.}
\end{figure}

We now describe the labels used for each part. The subsets of dimensions 3, 2, 1 and 0, of the partition of the parameter space will be denoted respectively by $V$, $S$, $L$ and $P$ for Volume, Surface, Line and Point, respectively. The surfaces are named using a number which corresponds to each bifurcation surface which is placed on the left side of the letter $S$. To describe the portion of the surface we place an index. The curves that are intersection of surfaces are named by using their corresponding numbers on the left side of the letter $L$, separated by a point. To describe the segment of the curve we place an index. Volumes and Points are simply indexed (since three or more surfaces may be involved in such an intersection).

We consider an example: the surface (${\cal S}_{1}$) splits into 6 different two--dimensional parts labeled from $1S_{1}$ to $1S_{6}$, plus some one--dimensional arcs labeled as $1.iL_{j}$ (where $i$ denotes the other surface intersected by (${\cal S}_{1}$) and $j$ is a number), and some zero--dimensional parts. In order to simplify the labels in Figs. \ref{sliceh0A_labels} and \ref{sliceh1A_purple_labels} we see \textbf{V1} which stands for the {\TeX} notation $V_{1}$. Analogously, \textbf{1S1} (respectively, \textbf{1.2L1}) stands for $1S_{1}$ (respectively, $1.2L_{1}$). And the same happens with other pictures.

All the bifurcation surfaces intersect on $h=0$. In fact, the equations of surfaces (${\cal S}_{1}$) and (${\cal S}_{2}$) are the parabolas $k+n^2=0$ and $4k+n^2=0$, respectively, when restricted to the plane $h=0$; the equations of (${\cal S}_{3}$), (${\cal S}_{5}$) and (${\cal S}_{8}$), restricted to $h=0$, are the straight lines $k=0$, $n=0$ and $k=0$, respectively; and surface (${\cal S}_{6}$) is the quartic $16k^{2}+24kn^{2}+n^{4}=0$ whose picture is the union of two parabolas with intersection at the origin. Finally, we note that all the elements above intersect at the origin of the bifurcation diagram when $h=0$.

As an exact drawing of the curves produced by intersecting the surfaces with slices gives us very small regions which are difficult to distinguish, and points of tangency are almost impossible to recognize, we have produced topologically equivalent pictures where regions are enlarged and tangencies are easy to observe. The reader may find the exact pictures in the web page {http://mat.uab.es/$\sim$artes/articles/qvfsn2SN02/}\linebreak{qvfsn2SN02.html.}

If we consider the value $h=1$, some changes in the bifurcation diagram happen. On the one hand, all the surfaces preserve their geometrical behavior, that is, surfaces (${\cal S}_{1}$) and (${\cal S}_{2}$) remain parabolas ($k-2n+n^{2}=0$ and $-4+4kn^{2}=0$, respectively); surfaces (${\cal S}_{3}$), (${\cal S}_{5}$) and (${\cal S}_{8}$) remain straight lines ($k-1=0$, $n-2=0$ and $k=0$, respectively); and surface (${\cal S}_{6}$) remains a quartic (with equation $-48+32k+16k^{2}+64n-64kn-24n^{2}+24kn^{2}+n^{4}=0$) whose picture is the union of two curves with intersection at the point $(1,0,2)$. On the other hand, compared to the case when $h=0$, there exist more intersection points among the surfaces and a new region appears between (${\cal S}_{3}$) and (${\cal S}_{8}$). All other regions, except this new ``middle'' one, remain topologically equivalent to the regions present in the case when $h=0$.

We recall that the black surface (${\cal S}_{6}$) (or $W_{4}$) means the turning of a finite antisaddle from a node to a focus. Then, according to the general results about quadratic systems, we could have limit cycles around such point.

\begin{remark} \label{rem:f-n} \rm
Wherever two parts of equal dimension $d$ are separated only by a part of dimension $d-1$ of the black bifurcation surface (${\cal S}_{6}$), their respective phase portraits are topologically equivalent since the only difference between them is that a finite antisaddle has turned into a focus without change of stability and without appearance of limit cycles. We denote such parts with different labels, but we do not give specific phase portraits in pictures attached to Theorems \ref{th:1.1} and \ref{th:1.2} for the parts with the focus. We only give portraits for the parts with nodes, except in the case of existence of a limit cycle or a graphic where the singular point inside them is portrayed as a focus. Neither do we give specific invariant description in Sec. \ref{sec:mainthm} distinguishing between these nodes and foci.
\end{remark}

\subsection{Bifurcation surfaces due to connections}

We now describe for each set of the partition on $h=1$ the local behavior of the flow around all the singular points. Given a concrete value of parameters of each one of the sets in this slice we compute the global phase portrait with the numerical program P4 \cite{Dumortier-Llibre-Artes:2006}. It is worth mentioning that many (but not all) of the phase portraits in this paper can be obtained not only numerically but also by means of perturbations of the systems of codimension one higher.

In this slice we have a partition in 2--dimensional regions bordered by curved polygons, some of them bounded, others bordered by infinity. Provisionally, we use low--case letters to describe the sets found algebraically so as not to interfere with the final partition described with capital letters. For each 2--dimensional region we obtain a phase portrait which is coherent with those of all their borders, except in one region. Consider the set $v_{1}$ in Fig. \ref{sliceh1A_nopurple}. In it we have only a saddle--node as finite singularity. When reaching the set $2.3\l_{1}$, we are on surfaces (${\cal S}_{2}$), (${\cal S}_{3}$) and (${\cal S}_{6}$) at the same time; this implies the presence of one more finite singularity (in fact, it is a cusp point) which is on the edge of splitting itself and give birth to finite saddle and antisaddle. Now, we consider the segments $2s_{2}$ and $2s_{3}$. By the Main Theorem of \cite{Vulpe:2011}, the corresponding phase portraits of these sets have a first--order weak saddle and a first--order weak focus, respectively. So, on $2s_{3}$ we have a Hopf bifurcation. This means that either in $v_{5}$ or $v_{10}$ we must have a limit cycle. In fact, it is in $v_{5}$. On the other hand, as we have a weak saddle on $2s_{2}$ and it is not detected a bifurcation surface intersecting this subset of loop type, neither its presence is forced to keep the coherence, its corresponding phase portrait is topologically equivalent to the portraits of $v_{4}$ and $v_{5}$. Since in $v_{5}$ we have a phase portrait topologically equivalent to the one on $2s_{2}$ (without limit cycles) and a phase portrait with limit cycles, this region must be split into two regions separated by a new surface (${\cal S}_{7}$) having at least one element $7S_{1}$ such that one region has limit cycle and the other does not, and the border $7S_{1}$ must correspond to a connection between separatrices. After numerical computations we check that it is the region $v_{5}$ the one which splits into $V_{5}$ without limit cycles and $V_{11}$ with one limit cycle.

The next result assures us the existence of limit cycle in any representative of the subset $v_{14}$ and it is needed to complete the study of $7S_{1}$.

\begin{lemma} \label{lemma_v14}
In $v_{14}$ there is always one limit cycle.
\end{lemma}

\begin{proof}
We see that the subset $v_{14}$ is characterized by $\mu<0$, $\mathcal{T}_{4}<0$, $W_{4}<0$, $M>0$, $\mathbb{T}>0$, $k>0$ and $n>0$. Any representative of $v_{14}$ has the finite saddle--node at the origin with its eigenvectors on the axes and two more finite singularities, a focus and a node (the focus is due to $W_{4}<0$). We claim that these two other singularities are placed in symmetrical quadrants with relation to the origin (see Fig. \ref{prop_v14}). In fact, by computing the exact expression of each singular point $(x_1,y_1)$ and $(x_2,y_2)$ and multiplying their $x$--coordinates and $y$-coordinates we obtain $k/{\mu}$ and $1/{\mu}$, respectively, which are always negative since $k>0$ and $\mu <0$ in $v_{14}$. Besides, each one of them is placed in an even quadrant since the product of the coordinates of each antisaddles is never null and any representative gives a negative product. Moreover, both antisaddles have the same stability since the product of their traces is given by $\mu/\mathcal{T}_{4}$ which is always positive in $v_{14}$.

\begin{figure}
\centering
\psfrag{V14}{$v_{14}$}
\centerline{\psfig{figure=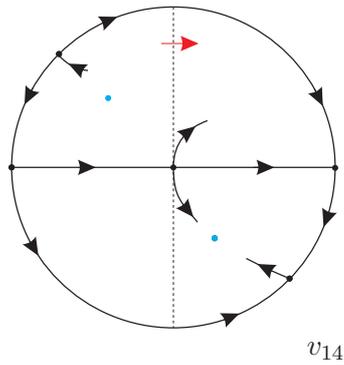,width=5cm}} \centerline {}
\caption{\small \label{prop_v14} The local behavior around each of the finite and infinite singularities of any representant of $v_{14}$. The red arrow shows the sense of the flow along the $y$-axis and the blue points are the focus and the node with same stability.}
\end{figure}

The infinite singularities of systems in $v_{14}$ are the saddle--node $\overline{\!{0\choose 2}\!\!}\ SN$ (recall the normal form \eqref{eqsnb}) and a saddle. In fact, the expression of the singular points in the local chart $U_{1}$ are $(0,0)$ and $((-2h+n)/k,0)$. We note that the determinant of the Jacobian matrix of the flow in $U_{1}$ at the second singularity is given by $$-\frac{(2 h-n)^2 \left(2 h n-k-n^2\right)}{k^2} = \frac{M^{2} \mu}{k^{2}},$$ which is negative since $\mu<0$ in $v_{14}$. Besides, this pair of saddles are in the second and the fourth quadrant because its first coordinate $(-2h+n)/k = -M/k$ is negative since $M>0$ and $k>0$ in $v_{14}$.

We also note that the flow along the $y$-axis is such that $\dot{x}>0$.

Since we have a pair of saddle points in the even quadrants, each of the finite antisaddles is in an even quadrant, no orbit can enter into the second quadrant and no orbit may leave the fourth one and, in addition, these antisaddles, a focus and a node, have the same stability, any phase portrait in $v_{14}$ must have at least one limit cycle in any of the even quadrants. Moreover, the limit cycle is in the second quadrant, because the focus is there since a saddle--node is born in that quadrant at $3s_{1}$, splits in two points when entering $v_{3}$ (both remain in the same quadrant since $x_{1}x_{2}=k/{\mu}<0$ and $y_{1}y_{2}=1/{\mu}<0$), the node turns into focus at $6s_{2}$ and the saddle moves to infinity at $1s_{2}$ appearing as node at the fourth quadrant when entering $v_{14}$. Furthermore, by the statement \eqref{item_iv} of Sec. \ref{sec:basicwf}, it follows the uniqueness of the limit cycle in $v_{14}$.
\end{proof}

Now, the following result states that the segment which splits the subset $v_{5}$ into the regions $V_{5}$ and $V_{11}$ has its endpoints well--determined.

\begin{proposition}
The endpoints of the surface $7S_{1}$ are $2.3\l_{1}$, intersection of surfaces (${\cal S}_{2}$) and (${\cal S}_{3}$), and $1.8\l_{1}$, intersection of surfaces (${\cal S}_{1}$) and (${\cal S}_{8}$).
\end{proposition}

Even though the next proof will be done in the concrete case when $h=1$, it can be easily extended for the generical case of surfaces and curves. We can visualize the image of this surface in the plane $h=1$ in Fig. \ref{sliceh1A_purple_labels}.

\begin{proof}
We consider $h=1$ and we write $r_{1}=(1,0,2)$ and $r_{2}=(1,1,0)$ for $2.3\l_{1}$ and $1.8\l_{1}$, respectively. If the starting point were any point of the segments $2s_{2}$ and $2s_{3}$, we would have the following incoherences: firstly, if the starting point of $7S_{1}$ were on $2s_{2}$, a portion of this subset must refer to a Hopf bifurcation since we have a limit cycle in $V_{11}$; and secondly, if this starting point were on $2s_{3}$, a portion of this subset must not refer to a Hopf bifurcation which contradicts the fact that on $2s_{3}$ we have a first--order weak focus. Finally, the ending point must be $r_{2}$ because, if it were located on $8s_{3}$, we would have a segment between this point and $1.8\l_{1}$ along surface (${\cal S}_{8}$) with two invariant straight lines and one limit cycle, which contradicts the statement \eqref{item_v} of Sec. \ref{sec:basicwf}, and if it were on $1s_{2}$, we would have a segment between this point and $1.8\l_{1}$ along surface (${\cal S}_{1}$) without limit cycle which is not compatible with Lemma \ref{lemma_v14} since $\mu=0$ does not produce a graphic.
\end{proof}

We show the sequence of phase portraits along these subsets in Fig. \ref{tripA}.

We cannot be totally sure that this is the unique additional bifurcation curve in this slice. There could exist others which are closed curves which are small enough to escape our numerical research, but the located one is enough to maintain the coherence of the bifurcation diagram. We recall that this kind of studies are always done modulo ``islands''. For all other two--dimensional parts of the partition of this slice whenever we join two points which are close to two different borders of the part, the two phase portraits are topologically equivalent. So we do not encounter more situations than the one mentioned above.

As we vary $h$ in $(0, \infty)$, the numerical research shows us the existence of the phenomenon just described, but for $h=0$, we have not found the same behavior.

In Figs. \ref{sliceh0A_labels} and \ref{sliceh1A_purple_labels} we show the complete bifurcation diagrams. In Sec. \ref{sec:mainthm} the reader can look for the topological equivalences among the phase portraits appearing in the various parts and the selected notation for their representatives in Figs. \ref{fig:phase_a1} and \ref{fig:phase_a2}.

\section{The bifurcation diagram of the systems in $Q\overline{sn}\overline{SN}(B)$} \label{sec:bifur_b}

We recall that, in view that the normal form \eqref{eqsnb} involves the coefficients $h$, $\l$ and $m$, which are real, the parameter space is $\R^3$ with coordinates $(h,\l,m)$.

Before we describe all the bifurcation surfaces for $Q\overline{sn}\overline{SN}(B)$, we prove the following result which gives conditions on the parameters for the presence of either a finite star node $n^*$ (whenever any two distinct non--trivial integral curves arrive at the node with distinct slopes), or a finite dicritical node $n^d$ (a node with identical eigenvalues but Jacobian non--diagonal).

\begin{lemma} \label{lemma_nodes}
Systems \eqref{eqsnb} always have a $n^*$, if $m=0$ and $h \neq 0$, or a $n^d$, otherwise.
\end{lemma}

\begin{proof} We note that the singular point $(0,-1/2h)$ has its Jacobian matrix given by
    \begin{equation*}
        \left(
            \begin{array}{cc}
                -1   & 0  \\
                -m/h & -1 \\
            \end{array}
        \right).
    \end{equation*}
\end{proof}

\subsection{Bifurcation surfaces due to the changes in the nature of singularities}

For systems \eqref{eqsnb} we will always have $(0,0)$ as a finite singular point, a double saddle--node. Besides, the needed invariants here are the same as in the previous system except surfaces (${\cal S}_{6}$) and (${\cal S}_{8}$); so that we shall only give the geometrical meaning and their equations plus a deeper discussion on surface (${\cal S}_{6}$). For further information about them, see Sec. \ref{sec:bifur_a}.

\onecolumn
\begin{figure}
\centering
\psfrag{V1}{$V_{1}$}   \psfrag{2.3L1}{$2.3L_{1}$}  \psfrag{6S2}{$6S_{2}$}
\psfrag{V4}{$V_{4}$}   \psfrag{2S2}{$2S_{2}$}      \psfrag{V5}{$V_{5}$}
\psfrag{7S1}{$7S_{1}$} \psfrag{XX}{$V_{11}$}       \psfrag{2S3}{$2S_{3}$}
\psfrag{X10}{$V_{10}$} \psfrag{6S3}{$6S_{3}$}      \psfrag{8S2}{$8S_{2}$}
\psfrag{V7}{$V_{7}$}   \psfrag{8S3}{$8S_{3}$}      \psfrag{V8}{$V_{8}$}
\psfrag{1S2}{$1S_{2}$} \psfrag{X14}{$V_{14}$}      \psfrag{1.8L1}{$1.8L_{1}$}
\psfrag{2S1}{$2S_{1}$} \psfrag{2S4}{$2S_{4}$}      \psfrag{1.2L1}{$1.2L_{1}$}
\centerline{\psfig{figure=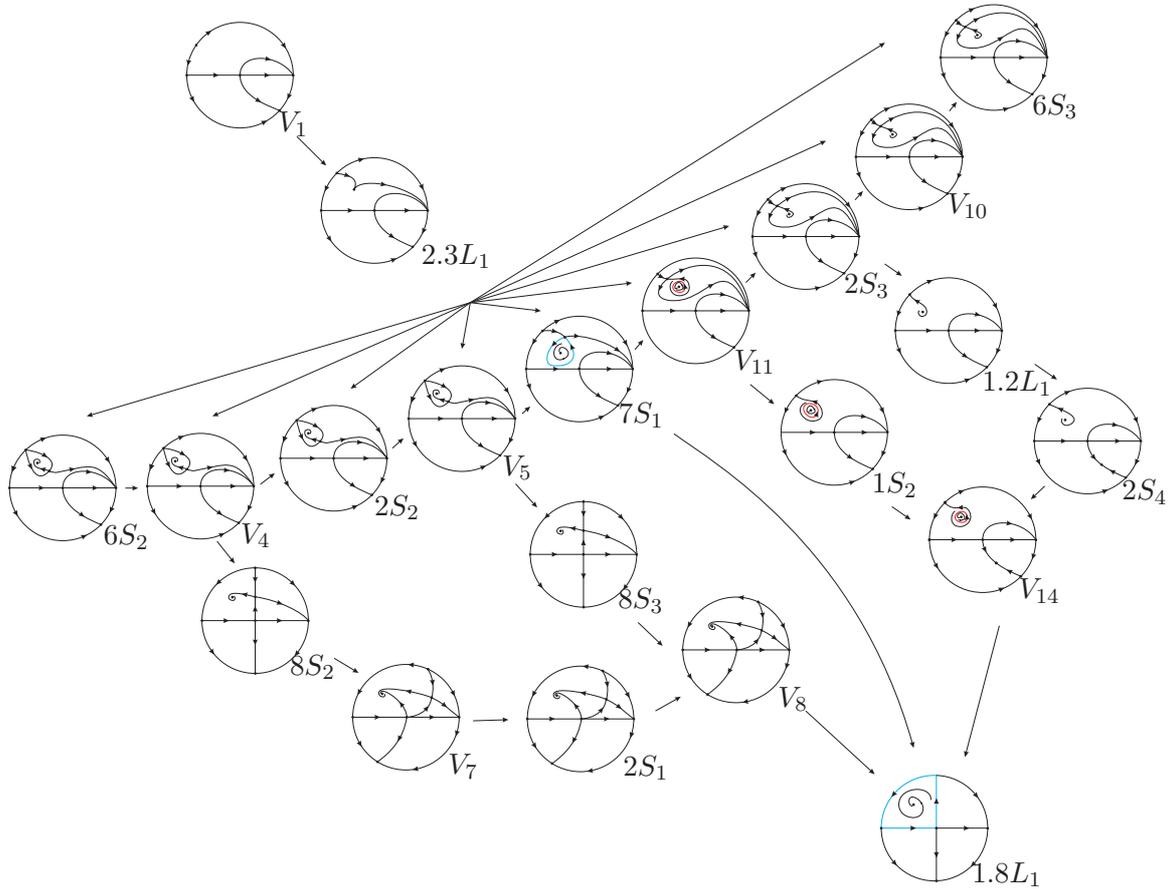,width=15cm}} \centerline {}
\caption{\small \label{tripA} Sequence of phase portraits in slice $h=1$ from $v_{1}$ to $1.8\l_{1}$. We start from $v_{1}$. When crossing $2.3\l_{1}$, we may choose at least seven ``destinations'': $6s_{2}$, $v_{4}$, $2s_{2}$, $v_{5}$, $2s_{3}$, $v_{10}$ and $6s_{3}$. In each one of these subsets, but $v_{5}$, we obtain only one phase portrait. In $v_{5}$ we find (at least) three different ones, which means that this subset must be split into (at least) three different regions whose phase portraits are $V_{5}$, $7S_{1}$ and $V_{11}$. And then we shall follow the arrows to reach the subset $1.8\l_{1}$ whose corresponding phase portrait is $1.8L_{1}$.}
\end{figure}
\twocolumn

\onecolumn
\begin{figure}
\centering
\psfrag{k}{$k$} \psfrag{n}{$n$}
\centerline{\psfig{figure=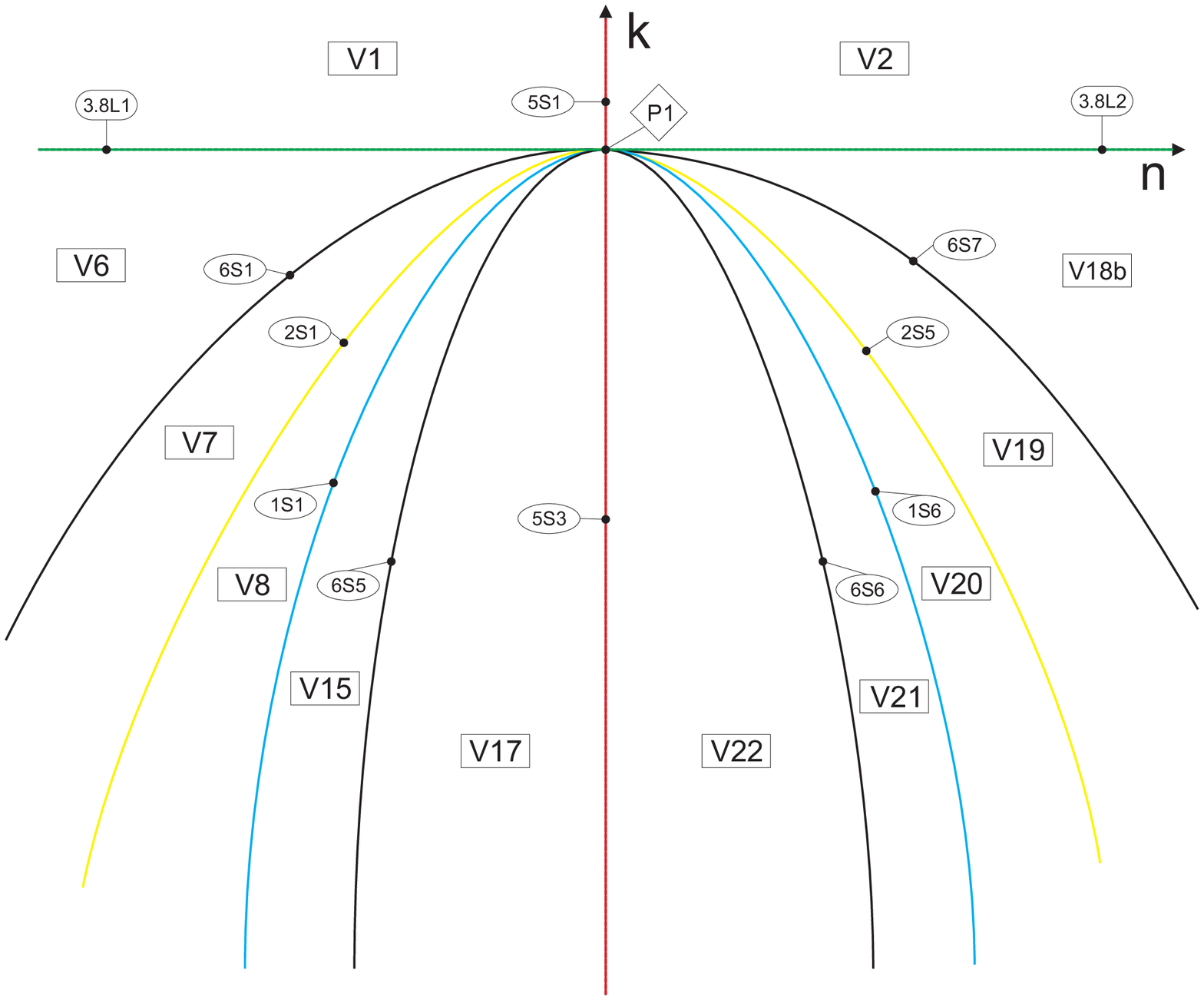,width=11cm}} \centerline {}
\caption{\small \label{sliceh0A_labels} Complete bifurcation diagram of $Q\overline{sn}\overline{SN}(A)$ for slice $h=0$.}
\end{figure}

\begin{figure}
\centering
\psfrag{k}{$k$} \psfrag{n}{$n$}
\centerline{\psfig{figure=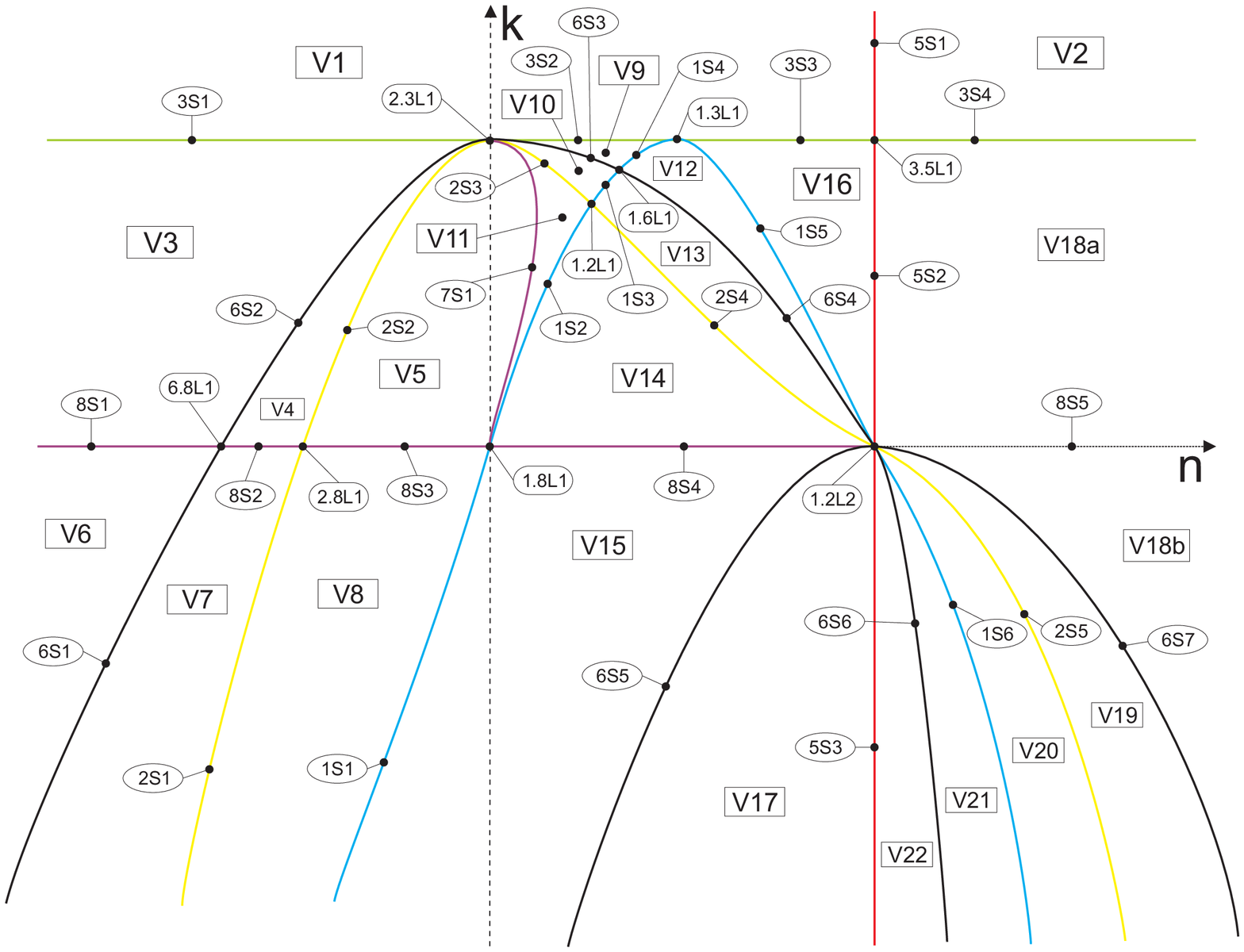,width=13cm}} \centerline {}
\caption{\small \label{sliceh1A_purple_labels} Complete bifurcation diagram of $Q\overline{sn}\overline{SN}(A)$ for slice $h=1$.}
\end{figure}
\twocolumn

%\medskip

\noindent \textbf{Bifurcation surfaces in $\R^3$ due to multiplicities of singularities}

\medskip

\noindent {\bf (${\cal S}_{1}$)} This is the bifurcation surface due to multiplicity of infinite singularities. This occurs when at least one finite singular point collides with at least one infinite point. The equation of this surface is $$\mu = 4h^{2}(1 + 2h\l - 2m) = 0.$$

\medskip

\noindent {\bf (${\cal S}_{3}$)} This is the bifurcation surface due to the collision of the other two finite singularities different from the saddle--node. The equation of this surface is given by $$\mathbb{T} = - h^2 = 0.$$ It only has substantial importance when we consider the plane $h=0$.

\medskip

\noindent {\bf (${\cal S}_{5}$)} This is the bifurcation surface due to the collision of infinite singularities, i.e. when all three infinite singular points collide. The equation of this surface is $$M = (-1 + 2m)^{2} = 0.$$

\medskip

\noindent \textbf{The surface of  $C^{\infty}$ bifurcation due to a strong saddle or a strong focus changing the sign of their traces (weak saddle or weak focus)}

\medskip

\noindent {\bf (${\cal S}_{2}$)} This is the bifurcation surface due to the weakness of finite singularities, which occurs when their trace is zero. The equation of this surface is given by $$\mathcal{T}_{4} = -16h^{3}\l = 0.$$

\medskip

\noindent \textbf{The surface of $C^{\infty}$ bifurcation due to a node becoming a focus}

\medskip

\noindent {\bf (${\cal S}_{6}$)} Since $W_{4}$ is identically zero for all the bifurcation space, the invariant that captures if a second point may be on the edge of changing from node to focus is $W_{3}$. The equation of this surface is given by $$W_{3} = 64h^{4}(1 + 2h{\l} + h^{2}{\l}^{2} - 2m) = 0.$$

\medskip

These are all the bifurcation surfaces of singularities of the systems \eqref{eqsnb} in the parameter space and they are all algebraic. We are not meant to discover any other bifurcation surface (neither non-algebraic nor algebraic one) due to the fact that in all the transitions we make among the parts of the bifurcation diagram of this family we find coherence in the phase portraits when ``traveling'' from one part to the other.

Analogously to the previous class, we shall foliate the 3--dimensional bifurcation diagram in $\R^3$ by planes $h=h_{0}$, $h_{0}$ constant. We shall give pictures of the resulting bifurcation diagram on these planar sections on an affine chart on $\R^2$. In order to detect the key values for this foliation, we must find the values of parameters where the surfaces have singularities and/or intersect to each other. We recall that we will be only interested in non-negative values of $h$ to construct the bifurcation diagram.

The following set of eleven results study the singularities of each surface and the simultaneous intersection points of the bifurcation surfaces, or the points or curves where two bifurcation surfaces are tangent.

We shall use the same set of colors for the bifurcation surfaces as in the previous case.

\begin{lemma}\label{lem:lemma1B} Concerning the singularities of the surfaces, it follows that:
    \begin{enumerate}[(i)]
        \item (${\cal S}_{1}$) has a straight line of singularities on $(0,\l,1/2), \ \l \in \R$;
        \item (${\cal S}_{2}$) has a straight line of singularities on $(0,0,m), \ m \in \R$;
        \item (${\cal S}_{3}$) and (${\cal S}_{5}$) have no singularities beyond their multiplicity as double planes;
        \item (${\cal S}_{6}$) has a straight line of singularities on $(0,\l,1/2), \ \l \in \R$.
    \end{enumerate}
\end{lemma}

\begin{proof} The gradient vector of surface (${\cal S}_{1}$) is given by $\nabla{\cal S}_{1}(h,\l,m) = (8h(1+3h\l-2m),8h^3,-8h^2)$, and solving the equation $\nabla{\cal S}_{1}(h,\l,m) = (0,0,0)$ we get that this surface has a straight line of singularities on $(0,\l,1/2), \ \l \in \R$, proving \emph{(i)}. The surface (${\cal S}_{2}$) trivially has a straight line of singularities on $(0,0,m), \ m \in \R$, which proves \emph{(ii)}. As the surfaces (${\cal S}_{3}$) and (${\cal S}_{5}$) are two double planes, they have no singularities; so \emph{(iii)} is proved. To prove \emph{(iv)}, we see that surface (${\cal S}_{6}$) is the product of a plane and a quartic, each one not having singularities itself. However, their intersection produces a straight line of singularities for this surface on $(0,\l,1/2), \ \l \in \R$.
\end{proof}

\begin{lemma}\label{lem:lemma2B} If $h=0$, the surfaces (${\cal S}_{1}$) and (${\cal S}_{2}$) coincide. For all $h \neq 0$, they intersect along the straight line $(h,0,1/2)$.
\end{lemma}

\begin{proof} By restricting both equations of surfaces (${\cal S}_{1}$) and (${\cal S}_{2}$) to $h=0$ they become both null, coinciding on the plane $h=0$. For all $h \neq 0$, by solving simultaneously both equations of these surfaces, we obtain the straight line $(h,0,1/2)$.
\end{proof}

\begin{lemma}\label{lem:lemma3B} The surfaces (${\cal S}_{1}$) and (${\cal S}_{3}$) coincide on the plane $h=0$ and have no intersection for all $h \neq 0$.
\end{lemma}

\begin{proof} By restricting both equations of surfaces (${\cal S}_{1}$) and (${\cal S}_{3}$) to $h=0$ they become both null, coinciding on the plane $h=0$, and it is easy to verify that both equations have no simultaneous solutions for all $h \neq 0$.
\end{proof}

\begin{lemma}\label{lem:lemma4B} If $h=0$, the surfaces (${\cal S}_{1}$) and (${\cal S}_{5}$) intersect along the straight line $(0,\l,1/2)$, $\l \in \R$. For all $h \neq 0$, they intersect along the straight line $(h,0,1/2)$.
\end{lemma}

\begin{proof} By restricting both equations of surfaces (${\cal S}_{1}$) and (${\cal S}_{6}$) to $h=0$ and solving the restricted equations simultaneously we obtain the straight line $(0,\l,1/2)$, $\l \in \R$, and by solving the system formed by the equations of these surfaces for all $h \neq 0$, we find the straight line $(h,0,1/2)$. Moreover, we see that both straight lines intersect at the point $(0,0,1/2)$.
\end{proof}

\begin{lemma}\label{lem:lemma5B} If $h=0$, the surfaces (${\cal S}_{1}$) and (${\cal S}_{6}$) coincide and, if $h \neq 0$, they have a $2$--order contact along the straight line $(h,0,1/2)$.
\end{lemma}

\begin{proof} By applying the same technique as in Lemma \ref{lem:lemma2B}, we have the result. To prove the $2$--order contact of the surfaces along the straight line $(h,0,1/2)$, $h \neq 0$, we note that $({\cal S}_{1})\cap({\cal S}_{6}) = \{(0,\l,m), \ \l,m \in \R\}\cup\{(h,0,1/2), \ \l \in \R\}$. Then, the gradient vector of both surfaces along the straight line $(h,0,1/2)$ is a multiple of the vector $(0,h,-1)$, $h \neq 0$. If we check the next derivative, they do not coincide, which proves the $2$--order contact.
\end{proof}

\begin{lemma}\label{lem:lemma6B} The surfaces (${\cal S}_{2}$) and (${\cal S}_{3}$) coincide on the plane $h=0$ and have no intersection for all $h \neq 0$.
\end{lemma}

\begin{proof} Analogous to Lemma \ref{lem:lemma3B}.
\end{proof}

\begin{lemma}\label{lem:lemma7B} The surfaces (${\cal S}_{2}$) and (${\cal S}_{5}$) intersect along the straight lines $(0,\l,1/2)$, $\l \in \R$ and $(h,0,1/2)$, $h \neq 0$.
\end{lemma}

\begin{proof} Analogous to Lemma \ref{lem:lemma4B}.
\end{proof}

\begin{lemma}\label{lem:lemma8B} The surfaces (${\cal S}_{2}$) and (${\cal S}_{6}$) coincide on the plane $h=0$ and intersect along the straight line $(h,0,1/2)$, $h \neq 0$.
\end{lemma}

\begin{proof} By restricting both equations of surfaces (${\cal S}_{2}$) and (${\cal S}_{6}$) to $h=0$ they become both null, coinciding on the plane $h=0$, and by solving the system formed by the equations of these surfaces for all $h \neq 0$, we find the straight line $(h,0,1/2)$.
\end{proof}

\begin{lemma}\label{lem:lemma9B} The surfaces (${\cal S}_{3}$) and (${\cal S}_{5}$) intersect along the straight line $(0,\l,1/2)$, $\l \in \R$.
\end{lemma}

\begin{proof} By solving the system formed by the equations of these surfaces we find the straight line $(0,\l,1/2)$, $\l \in \R$.
\end{proof}

\begin{lemma}\label{lem:lemma10B} The surfaces (${\cal S}_{3}$) and (${\cal S}_{6}$) coincide on the plane $h=0$ and have no intersection for all $h \neq 0$.
\end{lemma}

\begin{proof} Analogous to Lemma \ref{lem:lemma3B}.
\end{proof}

\begin{lemma}\label{lem:lemma11B} The surfaces (${\cal S}_{5}$) and (${\cal S}_{6}$) intersect along the straight lines $(0,\l,1/2)$, $\l \in \R$, and $(h,0,1/2)$, $h \in \R$, and the curve $(-2/{\l},\l,1/2)$, $\l \neq 0$.
\end{lemma}

\begin{proof} By solving the system formed by the equations of these surfaces we find the straight lines $(0,\l,1/2)$, $\l \in \R$, and $(h,0,1/2)$, $h \in \R$, and the curve $(-2/{\l},\l,1/2)$, $\l \neq 0$.
\end{proof}

Now, we shall study the bifurcation diagram having as reference the values of $h$ where significant phenomena occur in the behavior of the bifurcation surfaces. As there is not any other critical value of $h$, except $h=0$, this is the only value where the behavior of the bifurcation surfaces changes critically. Recalling we are considering only non-negative values of $h$, we shall choose a positive value to be a generic case.

We take then the values:
    \begin{equation}
        \aligned
            h_{0} & = 0,\\
            h_{1} & = 1.\\
        \endaligned
    \end{equation}

The value $h_{0}$ correspond to an explicit value of $h$ for which there is a bifurcation in the behavior of the systems on the slices. The value $h_{1}$ is just an intermediate point we call by a generic value of $h$ (see Figs. \ref{sliceh0B} and \ref{sliceh1B}). We recall that the bifurcation $\mathbb{T}=0$ (two other finite points collide), plotted usually in green, does not appear here because it corresponds completely to the plane $h=0$.

\begin{figure}
\centering
\psfrag{l}{$\l$} \psfrag{m}{$m$}
\centerline{\psfig{figure=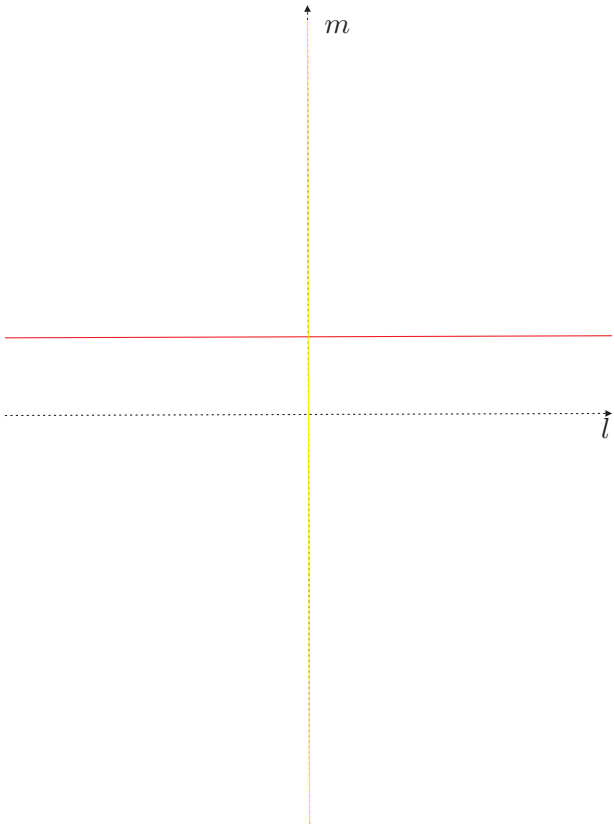,width=8cm}} \centerline {}
\caption{\small \label{sliceh0B} Slice of parameter space for \eqref{eqsnb} when $h=0$.}
\end{figure}

\begin{figure}
\centering
\psfrag{l}{$\l$} \psfrag{m}{$m$}
\centerline{\psfig{figure=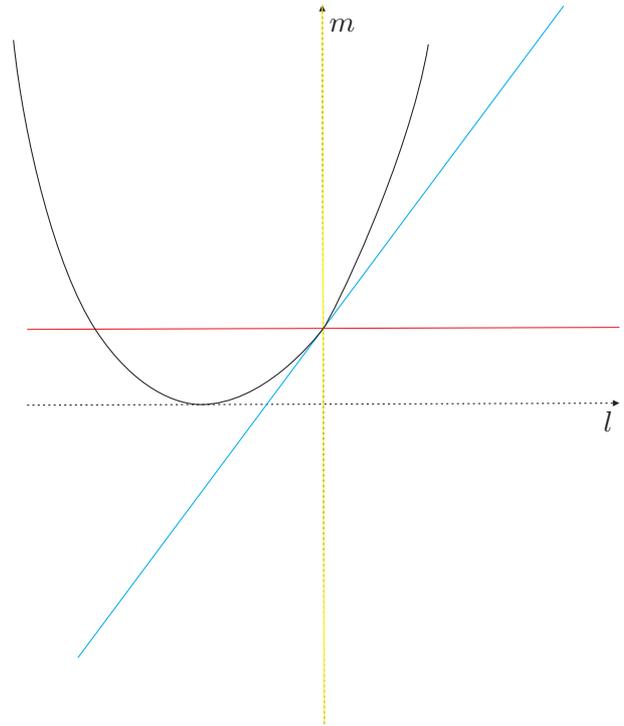,width=8cm}} \centerline {}
\caption{\small \label{sliceh1B} Slice of parameter space for \eqref{eqsnb} when $h=1$.}
\end{figure}

As in the previous family, all the bifurcation surfaces intersect on $h=0$. In fact, the equations of surfaces (${\cal S}_{1}$), (${\cal S}_{2}$), (${\cal S}_{3}$) and (${\cal S}_{6}$) are identically zero when restricted to the plane $h=0$, and the equation of (${\cal S}_{5}$) is the straight line $-1+2m=0$, for all $h \geq 0$ and $m,\l \in \R$.

Here we also give topologically equivalent figures to the exact drawings of the bifurcation curves. The reader may find the exact pictures in the web page {http://mat.uab.es/$\sim$artes/articles/qvfsn2SN02/}\linebreak{qvfsn2SN02.html.}

If we consider the value $h=1$, other changes in the bifurcation diagram happen. On this plane, surface (${\cal S}_{1}$) is the straight line $1+2\l-2m=0$, which intersects surface (${\cal S}_{5}$) at the point $(1,0,1/2)$; surface (${\cal S}_{6}$) is the parabola $1+2\l+\l^{2}-2m=0$ passing through the point $(1,0,1/2)$ with a $2$--order contact with surface (${\cal S}_{1}$), see Lemma \ref{lem:lemma5B}; moreover, Lemma \ref{lem:lemma11B} assures that surface (${\cal S}_{6}$) has another intersection point with surface (${\cal S}_{5}$) at $(1,-2,1/2)$; surface (${\cal S}_{2}$) is the straight line $\l=0$, which intersects surfaces (${\cal S}_{1}$), (${\cal S}_{5}$) and (${\cal S}_{6}$) at the point $(1,0,1/2)$, see Lemmas \ref{lem:lemma7B} and \ref{lem:lemma8B}; and finally, surface (${\cal S}_{3}$) is a negative constant.

We recall that the black surface (${\cal S}_{6}$) (or $W_{3}$) means the turning of a finite antisaddle from a node to a focus. Then, according to the general results about quadratic systems, we could have limit cycles around such focus for any set of parameters having $W_{3} < 0$.

In Figs. \ref{sliceh0B_labels} and \ref{sliceh1B_labels} we show the complete bifurcation diagrams. In Sec. \ref{sec:mainthm} the reader can look for the topological equivalences among the phase portraits appearing in the various parts and the selected notation for their representatives in Fig. \ref{fig:phase_b}.

\begin{figure}
\centering
\psfrag{l}{$\l$} \psfrag{m}{$m$}
\centerline{\psfig{figure=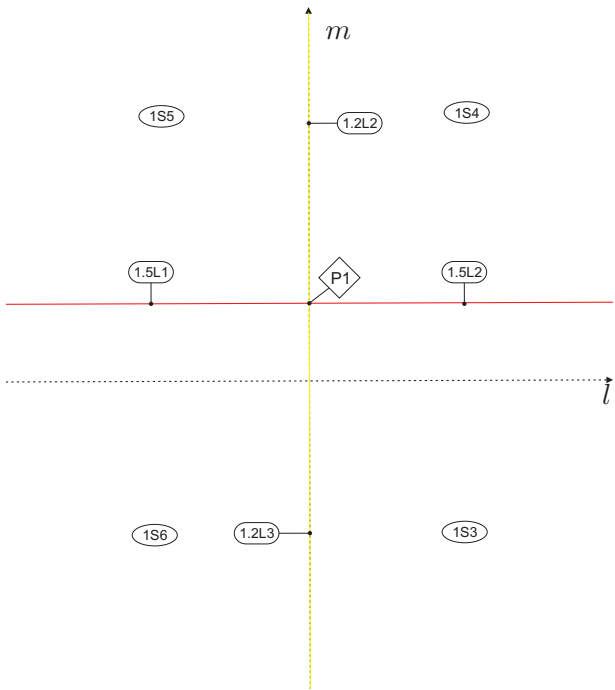,width=8cm}} \centerline {}
\caption{\small \label{sliceh0B_labels} Complete bifurcation diagram of $Q\overline{sn}\overline{SN}(B)$ for slice $h=0$.}
\end{figure}

\begin{figure}
\centering
\psfrag{l}{$\l$} \psfrag{m}{$m$}
\centerline{\psfig{figure=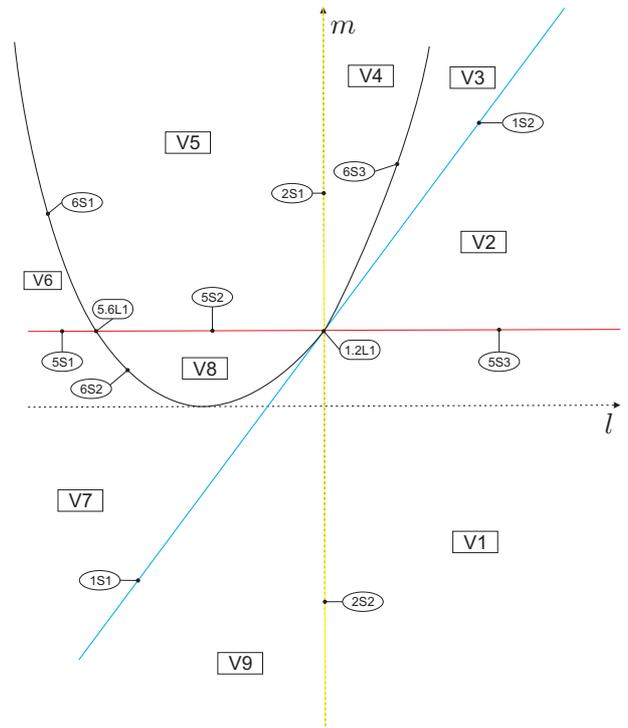,width=8cm}} \centerline {}
\caption{\small \label{sliceh1B_labels} Complete bifurcation diagram of $Q\overline{sn}\overline{SN}(B)$ for slice $h=1$.}
\end{figure}

\section{Other relevant facts about the bifurcation diagrams} \label{sec:islands}

The bifurcation diagrams we have obtained for the families $Q\overline{sn}\overline{SN}(A)$ and $Q\overline{sn}\overline{SN}(B)$ are completely coherent, i.e., in each family, by taking any two points in the parameter space and joining them by a continuous curve, along this curve the changes in phase portraits that occur when crossing the different bifurcation surfaces we mention can be completely explained.

Nevertheless, we cannot be sure that these bifurcation diagrams are the complete bifurcation diagrams for $Q\overline{sn}\overline{SN}(A)$ and $Q\overline{sn}\overline{SN}(B)$ due to the possibility of ``islands'' inside the parts bordered by unmentioned bifurcation surfaces. In case they exist, these ``islands'' would not mean any modification of the nature of the singular points. So, on the border of these ``islands'' we could only have bifurcations due to saddle connections or multiple limit cycles.

In case there were more bifurcation surfaces, we should still be able to join two representatives of any two parts of the 66 parts of $Q\overline{sn}\overline{SN}(A)$ or the 30 parts of $Q\overline{sn}\overline{SN}(B)$ found until now with a continuous curve either without crossing such bifurcation surface or, in case the curve crosses it, it must do it an even number of times without tangencies, otherwise one must take into account the multiplicity of the tangency, so the total number must be even. This is why we call these potential bifurcation surfaces ``\textit{islands}''.

However, in none of the two families we have found a different phase portrait which could fit in such an island. The existence of the invariant straight line avoids the existence of a double limit cycle which is the natural candidate for an island (recall the item \eqref{item_iv} of Sec. \ref{sec:basicwf}), and also the limited number of separatrices (compared to a generic case) limits greatly the possibilities for phase portraits.

\section{Completion of the proofs of Theorems \ref{th:1.1} and \ref{th:1.2}} \label{sec:mainthm}

\indent In the bifurcation diagram we may have topologically equivalent phase portraits belonging to distinct parts of the parameter space. As here we have finitely many distinct parts of the parameter space, to help us identify or to distinguish phase portraits, we need to introduce some invariants and we actually choose integer--valued invariants. All of them were already used in \cite{Llibre-Schlomiuk:2004,Artes-Llibre-Schlomiuk:2006}. These integer--valued invariants yield a classification which is easier to grasp.

\begin{definition} \rm We denote by $I_{1}(S)$ the number of the real finite singular points. This invariant is also denoted by $\N_{\R,f}(S)$ \cite{Artes-Llibre-Schlomiuk:2006}.
\end{definition}

\begin{definition}\rm We denote by $I_{2}(S)$ the sum of the indices of the real finite singular points. This invariant is also denoted by $\deg (DI_f(S))$ \cite{Artes-Llibre-Schlomiuk:2006}.
\end{definition}

\begin{definition}\rm We denote by $I_{3}(S)$ the number of the real infinite singular points. This number can be $\infty$ in some cases. This invariant is also denoted by $\N_{\R,\infty}(S)$ \cite{Artes-Llibre-Schlomiuk:2006}.
\end{definition}

\begin{definition} \rm We denote by $I_{4}(S)$ the sequence of digits (each one ranging from 0 to 4) such that each digit describes the total number of global or local separatrices (different from the line of infinity) ending (or starting) at an infinite singular point. The number of digits in the sequences is 2 or 4 according to the number of infinite singular points. We can start the sequence at anyone of the infinite singular points but all sequences must be listed in a same specific order either clockwise or counter--clockwise along the line of infinity.
\end{definition}

In our case we have used the clockwise sense beginning from the saddle--node at the origin of the local chart $U_{1}$ in the pictures shown in Figs. \ref{fig:phase_a1} and \ref{fig:phase_a2}, and the origin of the local chart $U_{2}$ in the pictures shown in Fig. \ref{fig:phase_b}.

\begin{definition} \rm We denote by $I_{5}(S)$ a digit which gives the number of limit cycles.
\end{definition}

As we have noted previously in Remark \ref{rem:f-n}, we do not distinguish between phase portraits whose only difference is that in one we have a finite node and in the other a focus. Both phase portraits are topologically equivalent and they can only be distinguished within the $C^1$ class. In case we may want to distinguish between them, a new invariant may easily be introduced.

\begin{definition} \rm We denote by $I_{6}(S)$ the digit 0 or 1 to distinguish the phase portrait which has connection of separatrices outside the straight line $\{y=0\}$; we use the digit 0 for not having it and 1 for having it.
\end{definition}

\begin{definition} \rm We denote by $I_{7}(S)$ the sequence of digits (each one ranging from 0 to 3) such that each digit describes the total number of global or local separatrices ending (or starting) at a finite antisaddle.
\end{definition}

\begin{theorem} \label{t11}
Consider the subfamily $Q\overline{sn}\overline{SN}(A)$ of all quadratic systems with a finite saddle--node $\overline{sn}_{(2)}$ and an infinite saddle--node of type $\overline{\!{0\choose 2}\!\!}\ SN$ located in the direction defined by the eigenvector with null eigenvalue. Consider now all the phase portraits that we have obtained for this family. The values of the affine invariant ${\cal I}=(I_{1},I_{2},I_{3},I_{4},I_{5},I_{6})$ given in the following diagram yield a partition of these phase portraits of the family $Q\overline{sn}\overline{SN}(A)$.

Furthermore, for each value of $\cal I$ in this diagram there corresponds a single phase portrait; i.e. $S$ and $S'$ are such that $I(S)=I(S')$, if and only if $S$ and $S'$ are topologically equivalent.
\end{theorem}

\begin{theorem} \label{t12}
Consider the subfamily $Q\overline{sn}\overline{SN}(B)$ of all quadratic systems with a finite saddle--node $\overline{sn}_{(2)}$ and an infinite saddle--node of type $\overline{\!{0\choose 2}\!\!}\ SN$ located in the direction defined by the eigenvector with non--null eigenvalue. Consider now all the phase portraits that we have obtained for this family. The values of the affine invariant ${\cal I}=(I_{1},I_{2},I_{3},I_{4},I_{7})$ given in the following diagram yield a partition of these phase portraits of the family $Q\overline{sn}\overline{SN}(B)$.

Furthermore, for each value of $\cal I$ in this diagram there corresponds a single phase portrait; i.e. $S$ and $S'$ are such that $I(S)=I(S')$, if and only if $S$ and $S'$ are topologically equivalent.
\end{theorem}

The bifurcation diagram for $Q\overline{sn}\overline{SN}(A)$ has 66 parts which produce 29 topologically different phase portraits as described in Table 1. The remaining 37 parts do not produce any new phase portrait which was not included in the 29 previous ones. The difference is basically the presence of a strong focus instead of a node and vice versa.

Similarly, the bifurcation diagram for $Q\overline{sn}\overline{SN}(B)$ has 30 parts which produce 16 topologically different phase portraits as described in Table 3. The remaining 14 parts do not produce any new phase portrait which was not included in the 16 previous ones. The phase portraits are basically different to each other under some algebro--geometric features related to the position of the orbits.

The phase portraits having neither limit cycle nor graphic have been denoted surrounded by parenthesis, for example $(V_{1})$ (in Tables 1 and 3); the phase portraits having one limit cycle have been denoted surrounded by brackets, for example $[V_{11}]$ (in Table 1); the phase portraits having one graphic have been denoted surrounded by $\{\}$, for example $\{7S_{1}\}$ (in Table 1).

\begin{proof}
The above two results follow from the results in the previous sections and a careful analysis of the bifurcation diagrams given in Secs. \ref{sec:bifur_a} and \ref{sec:bifur_b}, in Figs. \ref{sliceh0A_labels}, \ref{sliceh1A_purple_labels}, \ref{sliceh0B_labels} and \ref{sliceh1B_labels}, the definition of the invariants $I_{j}$ and their explicit values for the corresponding phase portraits.
\end{proof}

We first make some observations regarding the equivalence relations used in this paper: the affine and time rescaling, $C^1$ and topological equivalences.

The coarsest one among these three is the topological equivalence and the finest is the affine equivalence. We can have two systems which are topologically equivalent but not $C^1$--equivalent. For example, we could have a system with a finite antisaddle which is a structurally stable node and in another system with a focus, the two systems being topologically equivalent but belonging to distinct $C^1$--equivalence classes, separated by the surface ${\cal S}_{6}$ on which the node turns into a focus.

In Table 2 (Table 3, respectively) we listed in the first column 29 parts (16 parts, respectively) with all the distinct phase portraits of Figs. \ref{fig:phase_a1} and \ref{fig:phase_a2} (Fig. \ref{fig:phase_b}, respectively). Corresponding to each part listed in column 1 we have in its horizontal block, all parts whose phase portraits are topologically equivalent to the phase portrait appearing in column 1 of the same horizontal block.

In the second column we have put all the parts whose systems yield topologically equivalent phase portraits to those in the first column, but which may have some algebro--geometric features related to the position of the orbits.

In the third (respectively, fourth, and fifth) column we list all parts whose phase portraits have another antisaddle which is a focus (respectively, a node which is at a bifurcation point producing foci close to the node in perturbations, a node--focus to shorten, and a finite weak singular point). In the sixth column of Table 1 we list all phase portraits which have a triple infinite singularity.

Whenever phase portraits appear on a horizontal block in a specific column, the listing is done according to the decreasing dimension of the parts where they appear, always placing the lower dimensions on lower lines.

\subsection {Proof of the main theorem}

The bifurcation diagram described in Sec. \ref{sec:bifur_a}, plus Table 1 of the geometrical invariants distinguishing the 29 phase portraits, plus Table 2 giving the equivalences with the remaining phase portraits lead to the proof of the main statement of Theorem \ref{th:1.1}. Analogously, we have the proof of Theorem \ref{th:1.2}, but considering the description in Sec. \ref{sec:bifur_b} and Tables 3 and 4.

To prove statements \emph{(c)} and \emph{(d)} of Theorem \ref{th:1.2} we recall the Main Theorem of \cite{Vulpe:2011} and verify that:
    \begin{enumerate}[(i)]
        \item  Any representative of $2S_{1}$ is such that $h>0$, $\l=0$ and $m>1/2$. Then, we have: $\mathcal{T}_{4}=0$, $\mathcal{T}_{3}=8h^{2}(-1+2m) \neq 0$, $\mathcal{T}_{3} \mathcal{F}=-8h^{4}(-1+2m)^{3}<0$, $\mathcal{F}_{1} = \mathcal{F}_{2} = \mathcal{F}_{3} \mathcal{F}_{4} = 0$, which imply that it has an integrable center $c$;
        \item Any representative of $2S_{2}$ is such that $h>0$, $\l=0$ and $m<1/2$. Then, we have: $\mathcal{T}_{4}=0$, $\mathcal{T}_{3}=8h^{2}(-1+2m) \neq 0$, $\mathcal{T}_{3} \mathcal{F}=-8h^{4}(-1+2m)^{3}>0$, $\mathcal{F}_{1} = \mathcal{F}_{2} = \mathcal{F}_{3} \mathcal{F}_{4} = 0$, which imply that it has an integrable saddle $\is$.
    \end{enumerate}

\onecolumn
\begin{center}
{\sc Table 1:} Geometric classification for the subfamily $Q\overline{sn}\overline{SN}(A)$.
{%\footnotesize
\[
I_{1}= \left\{ \ba{ll}
%START OF BLOCK I_{1}=3
\mbox{$3$ \& } I_{2} = \left\{ \ba{l}
%start of block I_{1}=3, I_{2}=2
\mbox{$2$ \& $I_{3}=$} \left\{ \ba{l}

%start of block I_{1}=3, I_{2}=2, I_{3}=2
\mbox{$2$ \& $I_{4}=$} \left\{ \ba{l}
%start of block I_{1}=3, I_{2}=2, I_{3}=2, I_{4}=2111
\mbox{$2111$ \& $I_{5}=$} \left\{ \ba{l}
\mbox{$1$} \,\, [V_{14}], \\
\mbox{$0$} \,\, (V_{12}), \ea \right. \\
%start of block I_{1}=3, I_{2}=2, I_{3}=2, I_{4}=1111
\mbox{$1111$ \& $I_{6}=$} \left\{ \ba{l}
\mbox{$1$} \,\, \{8S_{4}\}, \\
\mbox{$0$} \,\, (V_{15}), \ea \right. \\
\ea \right. \\
%end of block I_{1}=3, I_{2}=2, I_{3}=2

%start of block I_{1}=3, I_{2}=2, I_{3}=1
\mbox{$1$} \, \, (5S_{3}) \\
%end of block I_{1}=3, I_{2}=2, I_{3}=1
%end of block I_{1}=3, I_{2}=2
\ea \right.\\

%start of block I_{1}=3, I_{2}=0
\mbox{$0$ \& $I_{3}=$} \left\{ \ba{l}

%start of block I_{1}=3, I_{2}=0, I_{3}=2
\mbox{$2$ \& $I_{4}=$} \left\{ \ba{l}
%start of block I_{1}=3, I_{2}=0, I_{3}=2, I_{4}=4111
\mbox{$4111$ \& $I_{5}=$} \left\{ \ba{l}
\mbox{$1$} \,\, [V_{11}], \\
\mbox{$0$} \,\, (V_{9}), \ea \right. \\
%start of block I_{1}=3, I_{2}=0, I_{3}=2, I_{4}=2111
\mbox{$2111$ \& $I_{6}=$} \left\{ \ba{l}
\mbox{$1$} \,\, (8S_{1}), \\
\mbox{$0$} \,\, (V_{6}), \ea \right. \\
\mbox{$3112$} \,\, (V_{3}), \\
\mbox{$2120$} \,\, (V_{16}), \\
\mbox{$3111$} \,\, \{7S_{1}\}, \\ \ea \right. \\
%end of block I_{1}=3, I_{2}=0, I_{3}=2

%start of block I_{1}=3, I_{2}=0, I_{3}=1
\mbox{$1$} \, \, (5S_{2}) \\
%end of block I_{1}=3, I_{2}=0, I_{3}=1
\ea \right.\\
%end of block I_{1}=3, I_{2}=0
\ea \right. \\

%START OF BLOCK I_{1}=2
\mbox{$2$ \& } I_{2} = \left\{ \ba{l}
%start of block I_{1}=2, I_{2}=1
\mbox{$1$ \& $I_{3}=$} \left\{ \ba{l}

%start of block I_{1}=2, I_{2}=1, I_{3}=infty
\mbox{$\infty$} \, \, (1.2L_{2}) \\
%end of block I_{1}=2, I_{2}=1, I_{3}=infty

%start of block I_{1}=2, I_{2}=1, I_{3}=2
\mbox{$2$ \& $I_{4}=$} \left\{ \ba{l}
%start of block I_{1}=2, I_{2}=1, I_{3}=2, I_{4}=2111
\mbox{$2111$ \& $I_{5}=$} \left\{ \ba{l}
\mbox{$1$} \,\, [1S_{2}], \\
\mbox{$0$} \,\, (1S_{4}), \ea \right. \\
%start of block I_{1}=3, I_{2}=0, I_{3}=2, I_{4}=1111
\mbox{$1111$ \& $I_{6}=$} \left\{ \ba{l}
\mbox{$1$} \,\, \{1.8L_{1}\}, \\
\mbox{$0$} \,\, (1S_{1}), \ea \right. \\
\mbox{$2110$} \,\, (1S_{5}), \\ \ea \right. \\
%end of block I_{1}=2, I_{2}=1, I_{3}=2
%end of block I_{1}=2, I_{2}=1
\ea \right.\\

%start of block I_{1}=2, I_{2}=0
\mbox{$0$ \& $I_{3}=$} \left\{ \ba{l}

%start of block I_{1}=2, I_{2}=0, I_{3}=2
\mbox{$2$ \& $I_{4}=$} \left\{ \ba{l}
%start of block I_{1}=2, I_{2}=0, I_{3}=2, I_{4}=2111
\mbox{$2111$ \& $I_{6}=$} \left\{ \ba{l}
\mbox{$1$} \,\, (3.8L_{1}), \\
\mbox{$0$} \,\, (3S_{4}), \ea \right. \\
\mbox{$3112$} \,\, (3S_{1}), \\
\mbox{$4111$} \,\, (3S_{2}), \\
\mbox{$2120$} \,\, (3S_{3}), \\
\mbox{$3111$} \,\, (2.3L_{1}), \\ \ea \right. \\
%end of block I_{1}=2, I_{2}=0, I_{3}=2

%start of block I_{1}=2, I_{2}=0, I_{3}=1
\mbox{$1$} \,\, (3.5L_{1}), \\
%end of block I_{1}=2, I_{2}=0, I_{3}=1
\ea \right.\\
%end of block I_{1}=2, I_{2}=0
\ea \right. \\

%START OF BLOCK I_{1}=1
\mbox{$1$ \& } I_{3} = \left\{ \ba{l}
%start of block I_{1}=2, I_{2}=1
%\mbox{$0$ \& $I_{3}=$} \left\{ \ba{l}

%start of block I_{1}=1, I_{2}=0, I_{3}=infty
\mbox{$\infty$} \, \, (P_{1}), \\
%end of block I_{1}=1, I_{2}=0, I_{3}=infty

%start of block I_{1}=1, I_{2}=0, I_{3}=2
\mbox{$2$} \, \, (V_{1}), \\
%end of block I_{1}=1, I_{2}=0, I_{3}=2

%start of block I_{1}=1, I_{2}=0, I_{3}=1
\mbox{$1$} \, \, (5S_{1}). \\
%end of block I_{1}=1, I_{2}=0, I_{3}=1
%end of block I_{1}=1, I_{2}=0
%\ea \right.\\
%end of block I_{1}=2, I_{2}=0
\ea \right. \\
    \ea \right.%
\]
}
\end{center}
\twocolumn

\onecolumn
\begin{center}
{\sc Table 2:} Topological equivalences for the subfamily $Q\overline{sn}\overline{SN}(A)$. \linebreak
\begin{tabular}{cccccc}
\hline
Presented & Identical       & Finite      & Finite       & Finite   & Triple \\
phase     & under           & antisaddle  & antisaddle   & weak     & infinite\\
portrait  & perturbations   & focus       & node--focus  & point    & point \\
\hline
\multirow{2}{*}{$V_{1}$} & $V_{2}$ & & & & \\
        & & & & & $1.3L_{1}$ \\
\hline
\multirow{2}{*}{$V_{3}$} & $V_{4}$, $V_{5}$ & & & & \\
        & & & $6S_{2}$ & $2S_{2}$ & \\
\hline
\multirow{2}{*}{$V_{6}$} & $V_{18a}$, $V_{18b}$ & $V_{7}$, $V_{8}$, $V_{19}$, $V_{20}$ & & &\\
        & $8S_{5}$ & & $6S_{1}$, $6S_{7}$ & $2S_{1}$, $2S_{5}$ &\\
\hline
\multirow{2}{*}{$V_{9}$} & & $V_{10}$ & & &\\
        & & & $6S_{3}$ & $2S_{3}$ &\\
\hline
$V_{11}$ & & & & &\\
\hline
\multirow{2}{*}{$V_{12}$} & & $V_{13}$ & & &\\
         & & & $6S_{4}$ & $2S_{4}$ &\\
\hline
$V_{14}$ & & & & &\\
\hline
\multirow{2}{*}{$V_{15}$} & $V_{21}$ & $V_{17}$, $V_{22}$ & & &\\
         & & & $6S_{5}$, $6S_{6}$ & &\\
\hline
$V_{16}$ & & & & &\\
\hline
$1S_{1}$ & $1S_{6}$ & & & &\\
\hline
$1S_{2}$ & & & & &\\
\hline
\multirow{2}{*}{$1S_{4}$} & & $1S_{3}$ & & &\\
         & & $1.6L_{1}$ & & $1.2L_{1}$ &\\
\hline
$1S_{5}$ & & & & &\\
\hline
$3S_{1}$ & & & & &\\
\hline
$3S_{2}$ & & & & &\\
\hline
$3S_{3}$ & & & & &\\
\hline
$3S_{4}$ & & & & &\\
\hline
$5S_{1}$ & & & & &\\
\hline
$5S_{2}$ & & & & &\\
\hline
$5S_{3}$ & & & & &\\
\hline
$7S_{1}$ & & & & &\\
\hline
\multirow{2}{*}{$8S_{1}$} & & $8S_{2}$, $8S_{3}$ & & &\\
         & & & $6.8L_{1}$ & $2.8L_{1}$ &\\
\hline
$8S_{4}$ & & & & &\\
\hline
$1.2L_{2}$ & & & & &\\
\hline
$1.8L_{1}$ & & & & &\\
\hline
$2.3L_{1}$ & & & & &\\
\hline
$3.5L_{1}$ & & & & &\\
\hline
$3.8L_{1}$ & $3.8L_{2}$ & & & &\\
\hline
$P_{1}$ & & & & &\\
\hline
\end{tabular}
\end{center}
\twocolumn

\onecolumn
\begin{center}
{\sc Table 3:} Geometric classification for the subfamily $Q\overline{sn}\overline{SN}(B)$.
{%\footnotesize
\[
I_{1}= \left\{ \ba{ll}
%START OF BLOCK I_{1}=3
\mbox{$3$ \& } I_{2} = \left\{ \ba{l}
%start of block I_{1}=3, I_{2}=2
\mbox{$2$ \& $I_{3}=$} \left\{ \ba{l}

%start of block I_{1}=3, I_{2}=2, I_{3}=2
\mbox{$2$ \& $I_{4}=$} \left\{ \ba{l}
%start of block I_{1}=3, I_{2}=2, I_{3}=2, I_{4}=1111
\mbox{$1111$ \& $I_{7}=$} \left\{ \ba{l}
\mbox{$32$} \,\, (V_{7}), \\
\mbox{$31$} \,\, (V_{6}), \\
\mbox{$20$} \,\, \{2S_{1}\}, \ea \right. \\
\mbox{$1121$} \, \, (V_{3}), \\
\ea \right. \\
%end of block I_{1}=3, I_{2}=2, I_{3}=2

%start of block I_{1}=3, I_{2}=2, I_{3}=1
\mbox{$1$} \, \, (5S_{1}) \\
%end of block I_{1}=3, I_{2}=2, I_{3}=1
%end of block I_{1}=3, I_{2}=2
\ea \right.\\

%start of block I_{1}=3, I_{2}=0
\mbox{$0$ \& $I_{3}=$} \left\{ \ba{l}

%start of block I_{1}=3, I_{2}=0, I_{3}=2
\mbox{$2$ \& $I_{4}=$} \left\{ \ba{l}
\mbox{$2120$} \,\, (V_{2}), \\
\mbox{$1121$} \,\, (V_{1}), \\ \ea \right. \\
%end of block I_{1}=3, I_{2}=0, I_{3}=2

%start of block I_{1}=3, I_{2}=0, I_{3}=1
\mbox{$1$} \, \, (5S_{3}) \\
%end of block I_{1}=3, I_{2}=0, I_{3}=1
\ea \right.\\
%end of block I_{1}=3, I_{2}=0
\ea \right. \\

%START OF BLOCK I_{1}=2
\mbox{$2$ \& } I_{2} = \left\{ \ba{l}
%start of block I_{1}=2, I_{2}=1
\mbox{$1$ \& $I_{3}=$} \left\{ \ba{l}

%start of block I_{1}=2, I_{2}=1, I_{3}=infty
\mbox{$\infty$} \, \, (1.2L_{1}) \\
%end of block I_{1}=2, I_{2}=1, I_{3}=infty

%start of block I_{1}=2, I_{2}=1, I_{3}=2
\mbox{$2$ \& $I_{4}=$} \left\{ \ba{l}
\mbox{$1120$} \,\, (1S_{2}), \\
\mbox{$1111$} \,\, (1S_{1}), \\
\ea \right. \\
%end of block I_{1}=2, I_{2}=1, I_{3}=2
%end of block I_{1}=2, I_{2}=1
\ea \right.\\
\ea \right. \\

%START OF BLOCK I_{1}=1
%\mbox{$1$} \& } I_{2} = \left\{ \ba{l}
%start of block I_{1}=1, I_{2}=0
\mbox{$1$ \& $I_{3}=$} \left\{ \ba{l}

%start of block I_{1}=1, I_{2}=0, I_{3}=infty
\mbox{$\infty$} \, \, (P_{1}), \\
%end of block I_{1}=1, I_{2}=0, I_{3}=infty

%start of block I_{1}=1, I_{2}=0, I_{3}=2
\mbox{$2$ \& $I_{4}=$} \left\{ \ba{l}
\mbox{$2121$} \, \, \{1S_{4}\}, \\
\mbox{$1111$} \, \, \{1.2L_{2}\}, \\
\mbox{$1011$} \, \, (1S_{3}), \\
\ea \right. \\
%end of block I_{1}=1, I_{2}=0, I_{3}=2

%start of block I_{1}=1, I_{2}=0, I_{3}=1
\mbox{$1$} \, \, \{1.5L_{1}\}. \\
%end of block I_{1}=1, I_{2}=0, I_{3}=1
%end of block I_{1}=1, I_{2}=0
\ea \right.\\
%end of block I_{1}=2, I_{2}=0
%\ea \right. \\
    \ea \right.%
\]
}
\end{center}
%\twocolumn
%
%\onecolumn
\begin{center}
{\sc Table 4:} Topological equivalences for the subfamily $Q\overline{sn}\overline{SN}(B)$. \linebreak
\begin{tabular}{ccccc}
\hline
Presented & Identical       & Finite      & Finite       & Finite \\
phase     & under           & antisaddle  & antisaddle   & weak   \\
portrait  & perturbations   & focus       & node--focus  & point \\
\hline
\multirow{2}{*}{$V_{1}$} & $V_{9}$ & & & \\
        & & & & $2S_{2}$ \\
\hline
$V_{2}$ & & & & \\
\hline
\multirow{2}{*}{$V_{3}$} & & $V_{4}$ & & \\
        & & & $6S_{3}$ & \\
\hline
\multirow{2}{*}{$V_{6}$} & & $V_{5}$ & & \\
        & & & $6S_{1}$ & \\
\hline
\multirow{2}{*}{$V_{7}$} & & $V_{8}$ & & \\
        & & & $6S_{2}$ & \\
\hline
$1S_{1}$ & & & & \\
\hline
$1S_{2}$ & & & & \\
\hline
\multirow{2}{*}{$1S_{3}$} & $1S_{6}$ & & & \\
         & & & & $1.2L_{3}$ \\
\hline
$1S_{4}$ & $1S_{5}$ & & & \\
\hline
$2S_{1}$ & & & & \\
\hline
\multirow{2}{*}{$5S_{1}$} & & $5S_{2}$ & & \\
         & & & $5.6L_{1}$ & \\
\hline
$5S_{3}$ & & & & \\
\hline
$3S_{1}$ & & & & \\
\hline
$1.2L_{1}$ & & & & \\
\hline
$1.2L_{2}$ & & & & \\
\hline
$1.5L_{1}$ & $1.5L_{2}$ & & & \\
\hline
$P_{1}$ & & & & \\
\hline
\end{tabular}
\end{center}
\twocolumn

%\bigskip

\noindent\textbf {Acknowledgements.} The first author is partially supported by a MEC/FEDER grant number MTM 2008--03437 and by a CICYT grant number 2005SGR 00550, the second author is supported by CAPES/DGU grant number BEX 9439/12--9 and the last author is partially supported by CAPES/DGU grant number 222/2010 and by FP7--PEOPLE--2012--IRSES--316338.

%
%
%
%%\addcontentsline{toc}{chapter}{\numberline{}Bibliography}
%%\bibliographystyle{amsalpha}
%%\bibliographystyle{plain}
%%\bibliography{ref}

\newcommand{\journal}[6]{#1 [#5] ``#2,'' \emph{#3} {\bf #4}, #6.}

\end{document}